\documentclass[11pt,a4paper]{article}
\usepackage{amsfonts}
\usepackage{amsmath,amssymb}
\usepackage{mathrsfs}
\usepackage{hyperref}
\usepackage{graphicx}
\usepackage{pdfsync}
\usepackage{amsthm}
\usepackage{txfonts}

\newtheorem{thm}{Theorem}[section]
\newtheorem{cor}[thm]{Corollary}
\newtheorem{lem}[thm]{Lemma}
\newtheorem{prop}[thm]{Proposition}
\newtheorem{defn}[thm]{Definition}

\newtheorem{rem}[thm]{Remark}

\numberwithin{equation}{section}\allowdisplaybreaks

\newcommand{\N}{\ensuremath{\mathbb{N}}}
\newcommand{\R}{\ensuremath{\mathbb{R}}}

\newcommand{\rt}{{\R^3}}

\def\no{\nonumber}

\begin{document}

\title{\Large\bf  Blowup criterion for Navier-Stokes equation in critical \\ Besov space with spatial dimensions $d \geq 4$ }

\author{\normalsize \bf   Kuijie Li$^{\dag}$, \ \ Baoxiang Wang$^{\dag ,}$\footnote{Corresponding Author}   \\
\footnotesize
\it $^\ddag$LMAM, School of Mathematical Sciences, Peking University, Beijing 100871, China, \\
\footnotesize
\it {Email: kuijiel@pku.edu.cn, \ wbx@pku.edu.cn }\\
 } \maketitle

\thispagestyle{empty}

\begin{abstract}
 This paper is concerned with the blowup criterion for mild solution to the incompressible Navier-Stokes equation in higher spatial dimensions $d \geq 4$.
 By establishing an $\epsilon$ regularity criterion in the spirit of~\cite{CKN}, we show that if the mild solution $u$ with initial data in $\dot B^{-1+d/p}_{p,q}(\R^d) $,\, $d<p,\,q<\infty$ becomes singular at a finite time $T_*$,  then
 \begin{align*}
  \limsup_{t\to T_*} \|u(t)\|_{\dot B^{-1+d/p}_{p,q}(\R^d)} = \infty.
 \end{align*}
The corresponding result in 3D case has been obtained in \cite{GKP}. As a by-product, we also prove a regularity criterion for the Leray-Hopf solution in the critical Besov space, which generalizes the results  in~\cite{DoDu09}, where blowup criterion in critical Lebesgue space $L^d(\mathbb{R}^d)$ is obtained.
\end{abstract}

\section{Introduction} \label{sec1}

In the present work, we consider the regularity problem of the solution to the incompressible Navier-Stokes equation (NS)
\begin{align} \label{ns}
\partial_t u - \Delta u + u\cdot \nabla u + \nabla p = 0, \ \ {\rm div}\,u=0, \ \ u(0,x) =u_0(x)
\end{align}
on $\R^d \times(0,T)$, where $d\geq 4$,\,  $u: \R^d \times (0,\,T) \to \R^d$ denotes the velocity vector field, $p$ represents the scalar pressure, $u_0$ is the initial data. NS plays a fundamental role in the fluid mechanics.

It is well-konwn that NS is scaling-invariant in the sense that if $u, p$ solves \eqref{ns} with data $u_0$,  so does   $u_{\lambda}(x,t)=\lambda u(\lambda x, \lambda^{2 }t)$ and $p_\lambda (x,t)=\lambda^2 p(\lambda x, \lambda^{2 }t)$ with initial data $\lambda u_0(\lambda x)$. A space $X$ defined on $\R^d$ is said to be critical provided that $\|u_0\|_{X} = \|\lambda u_0(\lambda \ \cdot)\|_{X}$ for any $\lambda>0$ (or more generally, $\|u_0\|_{X} \sim \|\lambda u_0(\lambda \ \cdot)\|_{X}$ and the equivalence is independent of $\lambda>0$), for example, $\dot H^{d/2-1}(\R^d),\,L^d(\R^d),\, \dot B^{s_p}_{p,q}(\R^d)$ are critical spaces, where
$$
s_p:= -1+\frac{d}{p}
$$
will be used in the whole paper (see Definiton~\ref{besovdefn} for Besov spaces).

In the poineering work~\cite{Le34},  J. Leray showed the existence of a global weak solution to the 3D Navier-Stokes equation defined on the whole space $\R^d$ with initial data in $L^2$, which is closely linked to the energy structure of the equation.  Later, Hopf~\cite{hopf51} extended this result to bounded smooth domain. The weak solution, which is now said to be the Leray-Hopf solution also exists in higher spatial dimensions, see Section~\ref{lerayhopfregu} for details.

The uniqueness and regularity of Leray-Hopf solution remains a long-standing open problem. However, various conditional results are obtained, for instance, the famous Ladyzhenskaya-Serrin-Prodi criterion, which asserts that a Leray-Hopf solution $u$ is regular and unique on $(0,T] \times \R^d$ if
\begin{align} \label{regucrit}
u \in L^q(0,T; L^r(\R^d)), \ \ \frac{2}{q} + \frac{d}{r} = 1, \ \  3\leq d \leq r \leq \infty.
\end{align}
The endpoint case  $r=d,\, q= \infty$ is much more subtle,  and it was  until 2003 that Escauriaza, Seregin and Sverak~\cite{EsSeSv03}  solved this endpoint case in 3D, later, Dong and Du~\cite{DoDu09} extended their result to the  case $d\geq 3 $. On the other hand, there are lots of interests in relaxing the condition~\eqref{regucrit}, for instance, Phuc showed the same conclusion for the 3D Leray-Hopf solution $u$ by assuming $u \in L^{\infty}(0,T; L^{3,m})$ with $3 \leq  m < \infty$, see~\cite{phuc}. Besides, according to~\cite{GKP,barker16}, the same result also applies for $u$ in 3D provided $u\in L^{\infty}(0,T;\dot B^{s_p}_{p,q})$ with $3<p,\,q<\infty$, a natural extension to the higer dimension in such setting is one of aims of our current paper. Finally, we mention a very interesting work, Buckmaster and Vicol~\cite{BuVi17} recently demonstrates a nonuniqueness result for the periodic weak solution in $\mathbb{T}^3$ with finite kinetic energy,  unfortunately, this weak solution  is   still not known as  a Leray-Hopf solution.

There is another way in constructing strong solution directly. It is well known that the Duhamel formula of~\eqref{ns} can be expressed as follows:
\begin{align} \label{NSI}
u(t) =   e^{ t \Delta } u_0 -  \int^t_0  e^{ (t-\tau) \Delta } \mathbb{P}\ \textrm{div}(u\otimes u)(\tau) d\tau,
\end{align}
where $\mathbb{P}=I-\nabla \Delta^{-1}{\rm div}$ is the
projection operator onto the divergence free vector fields. The solution to \eqref{NSI} is called a  mild solution  or strong solution. Kato and Fujita~\cite{KaFu62} initiated the study of~\eqref{NSI} in a fully invariant functional setting by using the semigroup method, which has a vivid perturbative feature and led to lots of results on various classes of (regular) solutions.  For example, let $d<p<q<\infty$,  Cannone~\cite{Can97}, Planchon~\cite{planchon} and Chemin~\cite{Ch99} used Kato's method to  derive the existence of  a unique local mild solution $u$ (in some properly chosen spaces) to NS with initial data in $\dot B^{s_p}_{p,q}$, see Theorem~\ref{subcribesov} in Section~\ref{proofregu} for more details. One can also refer to~\cite{Gi86,GiMi85,GiMi89,cal90,Ta92,We80,CaMe95,Wa06} and references therein for the local Cauchy theory in  Lebesgue space, Morrey space and others. It is known that NS is ill-posed in all critical Besov spaces $\dot B^{-1}_{\infty, q} (\mathbb{R}^d)$ with $d\geq 2, \ q\in [1,\infty]$ (cf. \cite{BoPa08, Ge08, Wa15, Yo10}) and up to now,  the known largest critical space for which NS is globally well posed for small initial data is $BMO^{-1}$, see Koch and Tataru \cite{KoTa01}.

Generally speaking, the mild solution associated with initial data in many critical spaces is not known to be global except for the small data solution.
This issue is complicated due to the lack of some uniform bounds in some spaces adapted to the NS with scaling invariance. On the opposite side, people turn to seek regularity criterion,  in other words, the blowup criterion. To be precise, let $X$ be a critical space, $u_0\in X$, assume $u$ is the mild solution with $u_0$ and the maximal existence time is denoted by $T_*$, whether or not the following assertion holds:
\begin{align} \label{blcriterion}
T_*< \infty \ \Rightarrow \  \limsup_{t\to T_*} \|u(t)\|_{X} = \infty.
\end{align}
Much progress has been made on this direction, Kenig and Koch proved the case $X= \dot H^{1/2}(\R^3)$ in~\cite{KeKo11}, afterwards, Gallagher, Koch and Planchon~\cite{GaKoPl13,GKP} further showed that~\eqref{blcriterion} is  also true for $X = L^3(\R^3) $ and $\dot B^{s_p}_{p,q}(\R^3) $ with $3<p,\,q<\infty$.
Besides, the upper limit in \eqref{blcriterion} can be refined as  a limit for $X = L^3(\rt),\, \dot B^{-1+3/p}_{p,q}(\rt)$, see Seregin~\cite{Seregin} and Dallas~\cite{albritton16} respectively, both of which employed a splitting argument and some type of weak solution. Motivated by the aforementioned results, we are led to consider whether~\eqref{blcriterion} holds for $X= \dot B^{s_p}_{p,q}(\R^d)$ with $4 \leq d <p,\,q<\infty$. Indeed, we shall answer it affirmably, see Theorem~\ref{mildblowupbesov} below.

Compared to those aimed at obtaining the global regularity of Leray-Hopf solution, another important aspect lies in founding partial regularity result
for weak solution satisfying local energy inequality. On that way, a key ingredient is the so-called $\epsilon$ regularity criterion.  Scheffer~\cite{sch76,sch77} started this way and got various  results for such weak solution in 3D. Inspired by Scheffer's results, L. Caffarelli, R. Kohn and L. Nirenberg in~\cite{CKN} exploited the best partial regularity result to date for the  suitable weak solution of the 3D Navier-Stokes equation.  Lin~\cite{Lin} gave a more direct and sketched proof of Caffarelli, Kohn and Nirenberg's result under the zero external force, for a detailed treatment, one can refer to~\cite{LS},  see also~\cite{vas07} for a De Giorgi method proof.  Recently, in papers~\cite{DoDu07} and~\cite{DoGu}, the authors showed a similar $\epsilon$ regularity criterion for the four dimensional NS in the context of classical solution and suitable weak solution respectively, thus leading to an estimate of the Hausdorff dimension of the singular set. By adapting the method in~\cite{vas07}, Wang and Wu~\cite{WaWu13} gave a unified proof of the partital regularity results for  NS  in the cases $d = 2, 3, 4$.

However, the notion of  suitable weak solution in dimension $d\geq 5$ needs to be slightly  modified (compared to the one in 3D or 4D)  so that the local energy inequality  makes sense, see Remark~\ref{leimakesense}, then an $\epsilon $ regularity criterion corresponding to such suitable weak solution can be  derived, which constitutes an integral part in proving the blowup result for solution in critical Besov space.

To introduce the suitable weak solution in higer dimension, let us specify the notion of  weak solution. Let $\Omega \subset \R^d$ be an open set,  $u,\,p$ is said to be a pair of  weak solution on $\Omega\times (0,T)$, provided  $u \in L^2_{loc}(\Omega \times (0,T)),\, p\in \mathcal{D}'(\Omega \times (0,T))$  satisfies NS in the sense of distributions. Hereafter,  the space dimension $d$, if not otherwise indicated,  is always assumed  to satisfy $d \geq 4$.


\begin{defn}[Suitable weak solution]\label{defsws}
Let $\Omega$ be an open set in $\R^d$, $Q:= \Omega \times (-T_1, T)$, $u\in L^{\infty}(-T_1, T; \dot B^{-1}_{\infty,\infty}(\R^d))$. $u,\, p$ is called a pair of suitable weak solution to~{\rm{(\ref{ns})}} on $Q$  if the following conditions are satisfied:
\begin{itemize}
\item[\rm(1)]
$u \in L^{\infty}\big((-T_1, T), L^2(\Omega)\big) \cap L^2(-T_1, T; H^1(\Omega))$, here $H^1$ denotes the usual Sobolev spaces, \
$p \in L^{{3/2}} (Q)$ ;

\item[\rm(2)] $u$, $p$ is a pair of weak solution  on $Q$;

\item[\rm(3)]
The following local energy inequality
\begin{align} \label{locaenerineq}
& \int_{\Omega} \varphi(x,t) |u(x,t)|^2 dx + 2 \int_{-T_1 }^t \int_{\Omega}\varphi(x,s) |\nabla u(x,s)|^2 dxds  \nonumber \\
 & \ \ \ \ \ \  \leq \int_{-T_1}^t \int_{\Omega}  |u|^2 (\partial_t \varphi + \Delta \varphi) + u\cdot \nabla \varphi (|u|^2 + 2p) dxds
\end{align}
holds for all $t \in (-T_1, T)$ and for all non-negative functions $\varphi \in C_0^{\infty}(\rt \times \mathbb{R})$ vanishing in a neighborhood of the parabolic boundary $\partial Q= \Omega \times \{t = -T_1\} \cup \partial \Omega \times (-T_1, T)$.
\end{itemize}
\end{defn}
\begin{rem} \label{leimakesense}
In the above definition,  $u \in L^{\infty}(-T_1,T;\dot B^{-1}_{\infty,\infty})$ is superfluous for $d =4$.  Indeed, it follows from the first item ${\rm(1)}$ that $u \in L^3_{loc}(Q)$,  thus each term in~\eqref{locaenerineq} makes sense. As for $d \geq 5$, based on our definition, one can obtain that $u \in L^4_{loc}(Q)$, see Proposition~\ref{stlffestimate}.
\end{rem}

Now we come to state our main result.
\begin{thm} \label{regularitycriterion}
Let $ u \in L^{\infty}(-1,0; \dot B^{-1}_{\infty,\infty}(\R^d))$ and  $u,\, p$ be a pair of suitable weak solution to {\rm NS} on $Q(1)$. Assume $ \|u\|_{L^{\infty}(-1,0; \dot B^{-1}_{\infty,\infty})} \leq M$, $ 1<\alpha <2$, then there exist constants $\epsilon_1$ and $C$, which depend on $M, d$ and $\alpha$ only, satisfying the following property. If
\begin{align}\label{smallhypo}
\sup_{-1<t<0}\int_{B(1)} |u|^2 dx + \int_{Q(1)}(|\nabla u|^2+|u|\, |p|) dxdt + \int_{-1}^0 \left(\int_{B(1)}|p|dx  \right)^{\alpha} dt \leq \epsilon_1,
\end{align}
Then
$$
\sup_{(x,t)\in Q(1/2)} |u(x,t)| \leq C.
$$
Here $Q(r):= B(r) \times (-r^2,0)$, $B(r) \subset \R^d$ denotes a ball centered at $0$ with radius $r$.
\end{thm}
\begin{rem}
The above conclusion is still valid without the assumption $u \in  L^{\infty}(-1,0; \dot B^{-1}_{\infty,\infty})$ for $d =4$ and $\alpha = 3/2$, one can refer to~\cite{WaWu13}.
\end{rem}

\begin{cor} \label{modifiedversionregu}
Let $r>0$,  $u \in L^{\infty}(-r^2,0;\dot B^{-1}_{\infty,\infty}(\R^d))$ and  $u,\, p$ forms a pair of suitable weak solution on $Q(r)$. Assume
  $ \|u\|_{L^{\infty}(-r^2,0; \dot B^{-1}_{\infty,\infty})} \leq M_r$, then there exist constants $\tilde{\epsilon}_1$ and $C$ relying only on $d$ and $M_r$, such that if
\begin{align} \label{scalesmaconst}
\frac{1}{r^{d-1}}\int_{Q(3r/4)} |u|^3 + |p|^{3/2} dx dt \leq \tilde{\epsilon}_1,
\end{align}
Then
\begin{align}
\sup_{(x,t)\in Q(r/4)} |u(x,t)| \leq \frac{C}{r}.
\end{align}
\end{cor}

Next we give our second main result concerning the regularity of mild solution with initial data in critical Besov spaces.
\begin{thm} \label{mildblowupbesov}
Assume $u_0 \in \dot B^{s_p}_{p,q}(\R^d)$, $ 4\leq d<p,\,q<\infty$. Let $u$ be the mild solution associated with  $u_0$, whose maximal existence time is $T_*$. If\, $T_* < \infty$, then necessarily
\begin{align}
\limsup_{t\to T_*} \|u(t)\|_{\dot B^{s_p}_{p,q}(\R^d)} = \infty.
\end{align}
\end{thm}
As a direct consequence, we have
\begin{cor} \label{leraybesov}
Let $d\geq 4$,\, $u$ be a Leray-Hopf solution to~\eqref{ns} on $Q_T:=\R^d \times (0,T) $ with $T< \infty$. Suppose further
\begin{align}
u \in L^{\infty}(0,T; \dot B^{s_p}_{p,q}(\R^d)),\ \ \ d< p,\,q < \infty.
\end{align}
Then $u$ is smooth and unique on $\R^d\times (0,T]$.

\end{cor}

Throughout out the paper,  $NS(u_0)$ represents the mild solution to~\eqref{NSI} with initial data $u_0$ and its maximal existence time
is denoted by $T(u_0)$.  Fix a point $z_0 = (x_0,t_0) \in \R^d \times \R$, $B(x_0,r)$ stands for  a ball centered at $x_0$ with radius $r$ and $B(r): = B(0,r)$.  Also, we have parabolic domain
 $$
 Q(z_0,r):= B(x_0,r)\times (t_0-r^2,t_0), \ \ \ Q(r): =B(r)\times (-r^2,0).
 $$
$\mathcal{S}$ and $\mathcal{S}'$ denote the Schwartz function class and tempered distribution respectively. For $f\in \mathcal{S}'$,  $\mathscr{F} f$ is the Fourier transform of $f$, and $\mathscr{F}^{-1}f$, the inverse Fourier transform of $f$.
The integral average of a function $u$ over some ball $B(x_0,r)$ is denoted by $[u]_{B(x_0,r)}$, i.e.
\begin{align}
[u]_{B(x_0,r)} =  \fint_{B(x_0,r)} u(x) dx = \frac{1}{|B(x_0,r)|} \int_{B(x_0,r)} u(x) dx.
\end{align}
In addition,  $ \|u\|_{L_t^qL_x^p(Q(r))} $ and $\|u\|_{\mathscr{K}^s_{p,\infty}(a,b)}$ mean
\begin{align}
\|u\|_{L_t^qL_x^p(Q(r))} & = \left( \int_{-r^2}^0 \|u(t)\|^q_{L_x^p(B(r))} dt \right)^{1/q}, \\
\|u\|_{\mathscr{K}^s_{p,\infty}(a,b)} & = \sup_{a<t<b} (t-a)^{-s/2}\|u(t)\|_{L_x^p(\R^d)}, \ \ s<0.
\end{align}
Specially, $\|u\|_{L^{p}(Q(r))} :=\|u\|_{L_t^pL_x^p(Q(r))} $ and  $\mathscr{K}^s_{p,\infty}(T):= \mathscr{K}^s_{p,\infty}(0,T)$ with $T>0$.  Various constants $C$ arise in the course of our work, they may different from line to line, $C(p,q,\ldots)$ or $C_{p,q,\ldots}$ means the constant depends on $p,\,q,\ldots$,  for simplicity, some  indices on which the constant $C$ relies are suppressed, as they are inessential for our argument. Finally, $p'$ is such that $1/p + 1/p' = 1$.

Let us conclude the introduction by giving the plan of the remaining sections. In Section~\ref{prepa}, we present some preliminary estimates, in particular, Lemma~\ref{estimateofcubic} and Lemma~\ref{diffpressureestimate}.  Section~\ref{pfepsilonregu} is devoted to the verification of Theorem~\ref{regularitycriterion} by using the ingredients in the previous section and Corollary~\ref{modifiedversionregu} is also showed in this part.  Theorem~\ref{mildblowupbesov} is demonstrated in Section~\ref{proofregu}, and the proof is divided into three parts; The regularity criterion for Leray-Hopf solution is given in the last section, where Corollary~\ref{leraybesov} is proved.


\section{Preliminary estimates} \label{prepa}

In this section, we present several results that play a major role in establishing the $\epsilon$ regularity criterion.  Let us first recall the definition of Besov spaces, in dimension $d \geq 1$, see~\cite{GKP}.  For a detailed presentation, one can also refer to~\cite{BaChDa11,Tr83,WaHuHaGu11}.
\begin{defn} \label{besovdefn}
Let $\varphi$ be a function  in $\mathcal{S}(\R^d)$ verifying $(\mathscr{F}{\varphi})(\xi) = 1 $ for $|\xi| \leq 1$ and $(\mathscr{F}{\varphi})(\xi) =0 $ for $|\xi|>2$, and denote $\varphi_j(x) : = 2^{jd}\varphi(2^j x)$,  then the frequency localization operators are defined by
$$
S_j : = \varphi_j *, \ \ \ \Delta_j : = S_{j+1} - S_j.
$$
Here $*$ is the convolution operator. A function $f \in \mathcal{S}'(\R^d)$ is said to belong to $\dot B^{s}_{p,q}= \dot B^s_{p,q}(\R^d)$ provided
\begin{itemize}
  \item[\rm (i)] the partial sum $\sum_{-m}^m \Delta_j f$ converges to $f$ as a tempered distribution if $s<d/p$ and after taking the quotient with polynomials if not, and
  \item[\rm (ii)] $ \|f\|_{\dot B^s_{p,q}} := \|2^{js}\|\Delta_j f\|_{L_x^p}\|_{\ell_j^q} < \infty.$
\end{itemize}
\end{defn}
The Besov space possesses many other equivalent characterizations, a particularly useful one in solving NS is given by the heat kernel. Indeed, we have (cf.~\cite{BaChDa11,Tr83})
\begin{align} \label{besovheat}
\|f\|_{\dot B^s_{p,q}} \sim \left (\int_0^{\infty} \big( t^{-s/2} \|e^{t\Delta}f\|_{L_x^p} \big)^q \frac{dt}{t} \right )^{1/q}, \ \ 1\leq p,\,q\leq \infty, \ \ s<0.
\end{align}
Here
\begin{align}
e^{t\Delta} f:= \mathscr{F}^{-1} e^{-t |\xi|^2} \mathscr{F} f(\xi).
\end{align}

The next interpolation inequality is borrowed from~\cite{BaChDa11}.
\begin{prop} \label{lffestimate}
Let $d\geq 1, 1 \leq q <p <\infty$ and $\alpha$ be a positive real number. There exists a constant $C$ such that
\begin{align}
\|f\|_{L^p(\R^d)} \leq C \|f\|^{1-\theta}_{\dot B^{-\alpha}_{\infty,\infty}(\R^d)} \|f\|^{\theta}_{\dot B^{\beta}_{q,q}(\R^d)} \ \  \textrm{with}\ \
\beta = \alpha \Big(\frac{p}{q}-1\Big) \ \  \textrm{and} \ \ \theta= \frac{q}{p}.
\end{align}
\end{prop}

\begin{prop}\label{besovestimate}
Let $d \geq 1$, $\phi \in \mathcal{S}(\R^d)$, there exists a constant $C$ such that
\begin{align}
\|\phi u \|_{\dot B^{-1}_{\infty,\infty}} \leq C (\|\phi\|_{L^{\infty}} + \|\phi\|_{\dot B^1_{d,1}}) \|u\|_{\dot B^{-1}_{\infty,\infty}}.
\end{align}
\end{prop}
\begin{proof}
We rewrite $\phi u$ as
\begin{align}
\phi u= T_u \phi + T_{\phi} u + R(u,\phi),
\end{align}
with
\begin{align}
T_u \phi = \sum_{j\in \mathbb{Z}} (S_{j-5} u) \Delta_j \phi, \ \ \ T_{\phi} u =\sum_{j\in \mathbb{Z}} (S_{j-5} \phi) \Delta_j u, \ \ \ R(u,\phi)= \sum_{|j-k| \leq 4 } \Delta_k u \Delta_j \phi.  \no
\end{align}
Due to the interaction of frequency, one can assert the existence of a positive constant $L$ so that
\begin{align}
\|\Delta_l T_u\phi\|_{L^{\infty}} \leq  C\sum_{|j-l|\leq L} \|S_{j-5}u \Delta_j \phi\|_{L^{\infty}} \leq C\sum_{|j-l|\leq L} \|S_{j-5}u \|_{L^{\infty}} \|\phi\|_{L^{\infty}}
\end{align}
where we have used the fact that $\Delta_j: L^p \rightarrow L^p$ is a bounded operator with $1 \leq p \leq \infty$. Another useful feature is that (cf. \cite{BaChDa11})
\begin{align}
\|f\|_{\dot B^s_{p,q}} \sim \|2^{js}\|S_j f\|_{L_x^p}\|_{\ell_j^q}, \ \  s<0,\ 1\leq p,\,q \leq \infty.
\end{align}
Thereby one can see
\begin{align}
\sup_{l\in \mathbb{Z}}2^{-l}\|\Delta_l T_u\phi\|_{L^{\infty}} \leq C_{L}\|u\|_{\dot B^{-1}_{\infty,\infty}} \|\phi\|_{L^{\infty}}.
\end{align}
The estimate of $T_{\phi} u$ is simpler, since $S_j$ is also bounded from $L^p$ to $L^p$ with $1 \leq p \leq \infty$, then
\begin{align}
\|\Delta_l T_{\phi}u\|_{L^{\infty}} \leq  C \sum_{|j-l|\leq L} \|S_{j-5}\phi \Delta_j u\|_{L^{\infty}} \leq C \sum_{|j-l|\leq L} \|\phi \|_{L^{\infty}} \|\Delta_j u \|_{L^{\infty}}.
\end{align}
Multiplying each side by $2^{-l}$, we can obtain
\begin{align}
\|T_{\phi} u\|_{\dot B^{-1}_{\infty,\infty}} \leq C_{L} \|\phi\|_{L^{\infty}}\|u\|_{\dot B^{-1}_{\infty,\infty}}.
\end{align}
Regarding to $R(u,\phi)$, we will estimate it in  $\dot B^{0}_{d,\infty}$ space, which is better as $\dot B^{0}_{d,\infty} \hookrightarrow \dot B^{-1}_{\infty,\infty}$.   For simplicity, we just consider a representative term  $\sum_{j\in \mathbb{Z}} \Delta_j u \Delta_j v$ in $R(u,\phi)$, since the argument for the others are almost the same.
Once again, there exists another positive constant $\tilde{L}$, such that
\begin{align}
\Big \|\Delta_l \Big(\sum_{j\in \mathbb{Z}} \Delta_j u \Delta_j \phi \Big)\Big\|_{L^{d}} \leq C \sum_{j \geq l-\tilde{L}} \|\Delta_j u \Delta_j \phi\|_{L^d}
\leq & C\sum_{j \geq l-\tilde{L}} \|\Delta_j u \|_{L^{\infty}} \|\Delta_j \phi\|_{L^d} \no \\
& \leq C \|u\|_{\dot B^{-1}_{\infty,\infty}} \|\phi\|_{\dot B^1_{d,1}}.
\end{align}
It turns out that the desired result holds if one collects  estimates for the three terms. The proof is finished.
\end{proof}

The local energy inequality~\eqref{locaenerineq} serves as a main tool to justify Theorem~\ref{regularitycriterion}. In higher spatial dimensions, one of the main difficulty arises in estimating $\int_{-T_1}^t \int_{\Omega} |u|^2 u\cdot \nabla \varphi dxd\tau$ in the right hand side of~\eqref{locaenerineq}, which is bounded by  $\|u\|_{L^3(\Omega\times (-T_1,t))}$. The following result is helpful to control this cubic term and in fact, if $\Omega$ is a ball,  a better local $L_t^4 L_x^4$ norm is obtained in terms of local energy under reasonable regularity assumption.
\begin{prop} \label{stlffestimate}
Let $d \geq 1$, $0< \gamma<1$, $\rho >0 $ and $a<b$. A constant $C(d,\gamma)$ depending only on $d, \gamma$  exists, such that
\begin{align}
\|u\|_{L^4(a,b;L^4(B(\gamma \rho)))}  \leq C(d,\gamma)&\|u\|^{1/2}_{L^{\infty}(a,b; \dot B^{-1}_{\infty,\infty}(\R^d))} \times  \no \\
&\bigg(\frac{\sqrt{b-a}}{\rho}\|u\|_{L^{\infty}(a,b;L^2(B(\rho)))} + \|\nabla u\|_{L^2(a,b;L^2(B(\rho)))}\bigg)^{1/2}.
\end{align}
\end{prop}
\begin{proof}
Choose $\phi \in C_0^{\infty}(\R^d)$ such that ${\rm{supp}}\, \phi \subset B(1)$, $0\leq \phi \leq 1$ and $\phi =1 $ on $B(\gamma)$. Set $\phi_{\rho}(x) = \phi(x/\rho)$, in view of
Proposition~\ref{lffestimate}, we see
\begin{align}
\|\phi_{\rho} u\|_{L^4(\R^d)} \leq C \|\phi_{\rho} u\|^{1/2}_{\dot B^{-1}_{\infty,\infty}(\R^d)}\|\phi_{\rho} u\|^{1/2}_{\dot H^1(\R^d)}.
\end{align}
Integrating in time, one can find
\begin{align} \label{spacetime}
\|\phi_{\rho} u\|_{L^4(a,b;L^4(\R^d))} \leq C \|\phi_{\rho} u\|^{1/2}_{L^{\infty}(a,b;\dot B^{-1}_{\infty,\infty}(\R^d))}\|\phi_{\rho} u\|^{1/2}_{L^2(a,b;\dot H^1(\R^d))}.
\end{align}
It is easy to see that
\begin{align} \label{h1estimate}
\|\phi_{\rho} u\|_{\dot H^1(\R^d)}\leq C \big(\rho^{-1}\|u\|_{L^2(B(\rho))} + \|\nabla u\|_{L^2(B(\rho))}\big).
\end{align}
On the other hand, applying Proposition~\ref{besovestimate}, we have
\begin{align}\label{scalebesov}
 \|\phi_{\rho} u\|_{\dot B^{-1}_{\infty,\infty}} &\leq C(\|\phi_{\rho}\|_{L^{\infty}}+ \|\phi_{\rho}\|_{\dot B^1_{d,1}}) \|u\|_{\dot B^{-1}_{\infty,\infty}} \no \\
 & \leq C(d,\gamma) \|u\|_{\dot B^{-1}_{\infty,\infty}}.
\end{align}
Inserting~\eqref{h1estimate} and~\eqref{scalebesov} into~\eqref{spacetime}, and noticing that
\begin{align}
\| u\|_{L^4(a,b;L^4(B(\gamma \rho)))} \leq \|\phi_{\rho} u\|_{L^4(a,b;L^4(\R^d))},
\end{align}
one can easily obtain the final result, as desired.
\end{proof}

Let $u, p$ be a pair of solution to~\eqref{ns},   we introduce some quantities involving  $u$ and $p$. Denote
\begin{align}
E_1 (r) &= \sup_{-r^2<t<0}\int_{B(r)}|u|^2 dx; \\
E_2(r) &= \int_{Q(r)} |\nabla u|^2 dxdt; \\
F(r) &= \int_{Q(r)} |u|\left||u|^2-[|u|^2]_{B(r)}\right| dxdt; \\
D(r) &= \int_{Q(r)} |u|^4 dxdt; \\
L(r) &=\int_{Q(r)} |u|\left|p-[p]_{B(r)}\right|dxdt; \\
K_{\alpha}(r) &= \int_{-r^2}^0 \left( \int_{B(r) } |p| dx \right)^{\alpha} dt.
\end{align}
 We remark that the above quantities follow from~\cite{CKN}, which are used for the control of the suitable weak solution $u$ in $Q(r)$. However, in \cite{CKN} they applied a version $F(r) = \int_{Q(r)} |u|^3 dxdt $ to show Theorem \ref{regularitycriterion} in 3D.   Noticing that for the suitable weak solution $u$ in $Q(r)$,  $\int_{Q(r)} u\cdot \nabla \varphi [|u|^2]_{B(r)} dxdt =0$, we have
\begin{align} \label{localerneimprov}
\int_{Q(r)} u\cdot \nabla \varphi |u|^2 dxdt = \int_{Q(r)} u\cdot \nabla \varphi (|u|^2-[|u|^2]_{B(r)}) dxdt
\end{align}
for the third term in the right hand side of local energy inequality \eqref{locaenerineq}. $u(|u|^2-[|u|^2]_{B(r)})$ enjoys a more delicate estimate than $|u|^3$, which is important for the estimates in higher spatial dimensions.

We are about to present two important lemmas in deriving Theorem~\ref{regularitycriterion}. Basically, they show how one can bound the right hand side of the local energy inequality.
\begin{lem} \label{estimateofcubic}
Let $d \geq 4$, then there exists a constant $C$, such that
\begin{align}
F(r) \leq C r^{2/d} E_1 (r)^{2/d} E_2 (r)^{1/2} D(r)^{1/2-1/d}.
\end{align}
\end{lem}
\begin{proof}
Let $m$ meet $1/m = 3/4 - 1/d $. By Sobolev's and Poincar\'e's inequalities, one sees
$$
\|f-[f]_{B(1)}\|_{L^m(B(1))} \leq C \|f-[f]_{B(1)}\|_{H^1_{4/3}(B(1))} \leq C \|\nabla f\|_{L^{4/3}(B(1))}.
$$
By scaling argument, we immediately have
\begin{align} \label{Poincare}
\|f-[f]_{B(r)}\|_{L^m(B(r))}  \leq C \|\nabla f\|_{L^{4/3}(B(r))}.
\end{align}
In view of H\"older's inequality and  \eqref{Poincare},
\begin{align} \label{2estimate1}
F(r) & \leq \int_{-r^2}^0 \|u\|_{L^{m'}(B(r))} \left\||u|^2-[|u|^2]_{B(r)} \right\|_{L^m(B(r))} dt \no \\
 & \leq C \int_{-r^2}^0 \|u\|_{L^{m'}(B(r))} \|\nabla |u|^2\|_{L^{4/3}(B(r))} dt  \no \\
 & \leq C \|u\|_{L_t^4L_x^{m'}(Q(r))} \|\nabla u \|_{L^2(Q(r))} \|u\|_{L^4(Q(r))}.
\end{align}
Using interpolation  inequality, we can find
\begin{align}
\|u\|_{L_x^{m'}(B(r))} \leq C \|u\|^{\theta}_{L_x^2(B(r))} \|u\|_{L_x^4(B(r))}^{1-\theta}, \ \ \theta = 4/d.
\end{align}
Integrating over time interval $(-r^2,0)$ and Using H\"older's inequality, one can see
\begin{align}\label{2estimate2}
\|u\|_{L_t^4L_x^{m'}(Q(r))} \leq C r^{\theta/2} \|u\|^{\theta}_{L_t^{\infty}L_x^2(Q(r))} \|u\|^{1-\theta}_{L^4(Q(r))}.
\end{align}
Inserting~\eqref{2estimate2} into~\eqref{2estimate1}, one can conclude the proof.
\end{proof}
With regard to the pressure $p$, one can observe that if $u, p$ satisfies NS distributionally on $\Omega \times (t_1,t_2)$ with $\Omega \subset \R^d$, then for a.e. $t\in (t_1,t_2)$,
\begin{align} \label{formulaofpre}
\Delta p = -\partial_i \partial_j (u^iu^j).
\end{align}
Here the summation convention over repeated indices is enforced. As in~\cite{CKN}, we localize $p$  to some bounded domain $\overline{\Omega'} \subset \Omega$. 
 let $\phi \in C_0^{\infty}(\Omega)$ be such that $\phi =1 $ on a neighborhood of $\overline{\Omega'}$, for $x\in \Omega'$, we have
\begin{align}
p = \phi p &= c_d \int_{\R^d} \frac{1}{|x-y|^{d-2}} \Delta_y (\phi p) dy \no \\
          & =c_d \int_{\R^d} \frac{1}{|x-y|^{d-2}}\big [ p\Delta_y \phi + \phi \Delta_y p + 2 \nabla_y \phi\cdot \nabla_y p \big] dy.
\end{align}
Putting~\eqref{formulaofpre} into the above formula and integrating by parts, one can obtain a useful expression for $\phi p$:
\begin{align}\label{expressionofpressure}
\phi p = \tilde{p} + p_3 + p_4,
\end{align}
where
\begin{align}
\tilde{p} &= c_d\int_{\R^d} \partial_{y_i} \bigg( \frac{1}{|x-y|^{d-2}}\bigg) \phi u^j \partial_{y_j} u^i  dy,  \no \\
p_3 &= -c_d(d-2) \int_{\R^d} \frac{x_j-y_j}{|x-y|^d} u^i u^j \partial_{y_i} \phi  dy -c_d \int_{\R^d} \frac{1}{|x-y|^{d-2}} u^i u^j \partial_{y_iy_j} \phi  dy, \no \\
p_4& = -c_d \int_{\R^d} \frac{1}{|x-y|^{d-2}} p\, \Delta_{y} \phi dy -2c_d(d-2) \int_{\R^d} \frac{x_j-y_j}{|x-y|^d} p\, \partial_{y_j} \phi  dy.
\end{align}

\begin{lem} \label{diffpressureestimate}
Let $d\geq 4,\, 0<r\leq   \rho/2$ and $1<\alpha<\infty $, then there exists  a constant $C$ depending on $d$, such that
\begin{align} \label{3estimate3}
L(r) & \leq Cr^{2/d}E_1 (r)^{2/d} E_2 (3r/2 )^{1/2} D (3r/2)^{1/2-1/d} \no \\
&\qquad + C r^{d/2+2} E_1 (r)^{1/2} \sup_{-r^2<t<0} \left(\int_{3r/2\leq|y|<\rho}  |u|^2 |y|^{-d} dy \right )^{1/2}   \left( \int_{-r^2}^0 \int_{3r/2 \leq|y|<\rho}  |\nabla u|^2 |y|^{-d}  dydt\right )^{1/2}  \no \\
& \qquad   + C \frac{r^{d/2+3}}{\rho^{d+1}} E_1 (r)^{1/2} A(\rho) + C\frac{r^{d/2+1+2/\alpha'}}{\rho^{d+1}} E_1 (r)^{1/2}K_{\alpha}(\rho)^{1/\alpha}.
\end{align}
\end{lem}
\begin{proof}
We use expression~\eqref{expressionofpressure} for $p$, where $\phi$ is chosen as follows:
\begin{align}
\phi = 1 \ \ &\textrm{on} \ \ B(3\rho/4),  \ \ \ \textrm{supp}\, \phi \subset B(\rho), \\
|\nabla \phi| &\leq C \rho^{-1}, \ \ \quad  |\nabla^2 \phi| \leq C\rho^{-2}.
\end{align}
Also, we further decompose $\tilde{p}$ into $\tilde{p} =  p_1 + p_2$, with
\begin{align}
p_1 &= c_d\int_{|y|< \frac{3}{2}r} \partial_{y_i} \bigg( \frac{1}{|x-y|^{d-2}}\bigg) \phi u^j \partial_{y_j} u^i  dy;  \label{p1express}\\
p_2 &= c_d\int_{\frac{3}{2}r \leq |y|<\rho} \partial_{y_i} \bigg( \frac{1}{|x-y|^{d-2}}\bigg) \phi u^j \partial_{y_j} u^i  dy.
\end{align}
Hence
\begin{align}
|p-[p]_{B(r)}| \leq \sum_{i=1}^4 |p_i - [p_i]_{B(r)}|.
\end{align}

For convenience, we denote
$$
L_i(r) =  \int_{Q(r)} |u|\left|p_i-[p_i]_{B(r)}\right|dxdt, \ i=1,...,4.
$$
For $x\in B(r)$, it can be easily verified that
\begin{align}
|\nabla p_2| &\leq C_d \int_{\frac{3}{2}r \leq |y|<\rho} |u||\nabla u|\frac{dy}{|y|^{d}},  \label{gradp2}\\
|\nabla p_3| &\leq \frac{C_d}{\rho^{d+1}} \int_{B(\rho)} |u|^2 dy, \ \ \ |\nabla p_4| \leq \frac{C_d}{\rho^{d+1}} \int_{B(\rho)} |p|dy. \label{gradp3}
\end{align}
On the other hand,
\begin{align}
T_{ij}(\psi) := \partial_{ij}\big(|x|^{-(d-2)} \big)* \psi \label{riesz}
\end{align}
is a Calderon-Zygmund operator, which is bounded from $L^p(\R^d)$ to itself for $1<p<\infty$. Let $m$ be
such that $1/m = 3/4 - 1/d$, we have from H\"older's inequality, \eqref{Poincare}, \eqref{p1express} that
\begin{align}
L_1(r) &\leq \int_{-r^2}^0 \|u\|_{L^{m'}(B(r))} \|p_1-[p_1]_{B(r)}\|_{L^m(B(r))} dt \no \\
& \leq C \int_{-r^2}^0 \|u\|_{L^{m'}(B(r))} \|\nabla p_1\|_{L^{4/3}(B(r))} dt  \no \\
& \leq C \int_{-r^2}^0 \|u\|_{L^{m'}(B(r))} \|u\nabla u\|_{L^{4/3}(B(3r/2))} dt.
\end{align}
Therefore, one can argue as Lemma~\ref{estimateofcubic} to obtain that
\begin{align} \label{pressure1}
L_1(r)  \leq C r^{2/d} \|u\|^{4/d}_{L_t^{\infty}L_x^2(Q(r))} \|\nabla u\|_{L^2(Q(3r/2))} \|u\|^{2-4/d}_{L^4(Q(3r/2))}.
\end{align}
For the estimate of $L_2(r)$, in view of H\"older's inequality, \eqref{gradp2} and mean value theorem, we have
\begin{align}
L_2(r) &\leq C r \int_{-r^2}^0 |\nabla p_2|_{L^{\infty}(B(r))}\bigg(\int_{B(r)}|u| dx \bigg)dt  \label{meanvalue} \\
& \leq C r^{\frac{d}{2}+1} \|u\|_{L_t^{\infty}L_x^2(Q(r))} \int_{-r^2}^0 \int_{\frac{3}{2}r \leq |y|<\rho} \frac{1}{|y|^d} |u||\nabla u| dy dt. \no
\end{align}
Using H\"older inequality again, one can see
\begin{align} \label{pressure2}
L_2(r)
& \leq C r^{\frac{d}{2}+2}\|u\|_{L_t^{\infty}L_x^2(Q(r))} \sup_{-r^2<t<0} \left(\int_{\frac{3}{2}r\leq|y|<\rho} \frac{1}{|y|^d} |u|^2 dy \right )^{\frac{1}{2}}  \no \\
&  \qquad    \times \left( \int_{-r^2}^0 \int_{\frac{3}{2}r\leq|y|<\rho} \frac{1}{|y|^d} |\nabla u|^2  dydt\right )^{\frac{1}{2}}.
\end{align}
To estimate $L_3(r)$, from \eqref{meanvalue}, H\"older's inequality  and \eqref{gradp3} it follows that
\begin{align} \label{pressure3}
L_3(r) &\leq C r \int_{-r^2}^0 |\nabla p_3|_{L^{\infty}(B(r))} \bigg(\int_{B(r)}|u| dx \bigg) dt  \no \\
 & \leq \frac{Cr}{\rho^{d+1}} \int_{-r^2}^0 \bigg(\int_{B(r)}|u|dx \bigg) \bigg(\int_{B(\rho)} |u|^2 dy\bigg) dt  \no  \\
 & \leq \frac{Cr^{\frac{d}{2}+3}}{\rho^{d+1}} \|u\|_{L_t^{\infty}L_x^2(Q(r))} \|u\|^2_{L_t^{\infty}L_x^2(Q(\rho))}.
\end{align}
Finally, for the estimate of $L_4(r)$, Using \eqref{meanvalue}, H\"older's inequality and \eqref{gradp3}, one has that
\begin{align} \label{pressure4}
L_4(r) & \leq C r \int_{-r^2}^0 |\nabla p_4|_{L^{\infty}(B(r))}\bigg(\int_{B(r)}|u| dx\bigg) dt \no \\
 & \leq C \frac{r}{\rho^{d+1}} \int_{-r^2}^0 \bigg(\int_{B(r)} |u| dx\bigg) \bigg(\int_{B(\rho)} |p| dy \bigg)dt \no \\
 & \leq C \frac{r^{d/2+1}}{\rho^{d+1}} \|u\|_{L_t^{\infty}L_x^2(Q(r))} \int_{-r^2}^0 \|p\|_{L^1(B(\rho))} dt  \no \\
 & \leq C\frac{r^{d/2+1+2/\alpha'}}{\rho^{d+1}} \|u\|_{L_t^{\infty}L_x^2(Q(r))} \|p\|_{L_t^{\alpha}L_x^1(Q(\rho))}.
\end{align}
Noticing that
$
L(r) \leq \sum_{i=1}^4 L_i(r)
$
and combining the results~\eqref{pressure1}, \eqref{pressure2}, \eqref{pressure3} and~\eqref{pressure4}, we can get the desired result.
\end{proof}


\section{$\epsilon$-regularity criterion } \label{pfepsilonregu}

In this section, we will use the same strategy as that in~\cite{CKN} to prove Theorem~\ref{regularitycriterion}, performing an induction on the local energy.  In fact, under the assumption of Theorem~\ref{regularitycriterion}, we shall show  that for each $z_0:=(a,s) \in Q(1/2)$,
\begin{align} \label{3estimate1}
\fint_{|x-a|<r_n} |u|^2(x,s) dx \leq C \epsilon_1, \ \ r_n =2^{-n}, \ \ \forall\, n \geq 2.
\end{align}
where $C$ is a constant that will be chosen suitably in our proof. Additionally, assume $z_0$ is a Lebesgue point for $u$, then~\eqref{3estimate1} implies
\begin{align} \label{boundedquan}
|u|^2(a,s) \leq C \epsilon_1,
\end{align}
hence almost everywhere in $Q(1/2)$.

Due to the translation invariance of the NS equation and the hypothesis in Theorem~\ref{regularitycriterion}, one can assume $z_0 = 0$ in the sequel.
To show~\eqref{3estimate1}, we will prove inductively that
\begin{align}
  (I)_n: & \quad \ \  \frac{1}{r_n^{d+1+2/d}} F(r_n) + \frac{1}{r_n^{d+1+\gamma_{d,\alpha}}} L(r_n) \leq \epsilon_1, \ \ \  n\geq 3,  \label{localFL}\\
(R)_n: &  \quad \ \ \ E_1(r_n) + E_2(r_n) \leq C\epsilon_1 r_n^d,  \ \ n\geq 2,  \label{localenergybound}
\end{align}
where $\gamma_{d,\alpha} = \min\{ 2/d, 2/\alpha'  \}$,  $C$ is a constant depending on $d,\, M$ and $\alpha$. Clearly, \eqref{localenergybound} implies~\eqref{3estimate1} with $(a,s)= (0,0)$. Next, we  show the validity of $(I)_n$ and $(R)_n$.

\begin{proof}[Proof of Theorem~\ref{regularitycriterion}] We will use the following way to show the results of $(I)_n$ and $(R)_n$:  (1) We show that $(R)_2$ holds; (2) $(R)_k$  holds for $2\leq k\leq n$ implies that $(I)_{n+1}$; (3) $(I)_k$ holds for $3 \leq k \leq n$ implies that $(R)_n$.  Then by induction we have \eqref{localFL} and \eqref{localenergybound}.

{\it Step 1}. We prove that $(R)_2$ holds.  Recalling that our hypotheses are
\begin{align}
\sup_{-1<t<0}\int_{B(1)} |u|^2 dx + &\int_{Q(1)}(|\nabla u|^2+|u||p|) dxdt + \int_{-1}^0 \Big(\int_{B(1)}|p|dx  \Big)^{\alpha} dt \leq \epsilon_1,  \label{smallconst} \\
& \|u\|_{L^{\infty}(-1,0;\dot B^{-1}_{\infty,\infty})} \leq M < \infty.
\end{align}
As a priori, assume $\epsilon_1 \leq 1$.  Apparently, for $C \geq r_2^{-d}$, one has that
\begin{align}
E_1 (r_2) + E_2 (r_2) \leq \epsilon_1 \leq C\epsilon_1 r_2^d,
\end{align}

{\it Step 2}. For all $n \geq 2$, we show that $(R)_k$ holds for $2\leq k\leq n$ implies the result of $(I)_{n+1}$. Note that our inductive hypothesis is
\begin{align}
E_1 (r_k) + E_2 (r_k) \leq C\epsilon_1 r_k^d, \ \ \forall\, 2\leq k \leq n.
\end{align}
One can easily see $E_1 (s_1) \leq E_1 (s_2)$ provided $0< s_1 \leq s_2$ and the same holds for $E_2 (r)$.   For the first term in $(I)_{n+1}$, by Lemma~\ref{estimateofcubic} and Proposition~\ref{stlffestimate}(set $\gamma = 1/2, \rho =r_{n}$), we have for any $n \geq 2$,
\begin{align} \label{cubicestimate1}
F(r_{n+1}) & \leq C r_{n+1}^{\frac{2}{d}} E_1 (r_{n+1})^{\frac{2}{d}} E_2 (r_{n+1})^{\frac{1}{2}} D(r_{n+1})^{\frac{1}{2}-\frac{1}{d}} \no \\
& \leq C r_n^{\frac{2}{d}}(C\epsilon_1 r_n^d)^{\frac{2}{d}}(C\epsilon_1 r_n^d)^{\frac{1}{2}}(C\epsilon_1 r_n^d)^{\frac{1}{2}-\frac{1}{d}} \no \\
& \leq C r_n^{d+1+\frac{2}{d}} \epsilon_1^{1+\frac{1}{d}}.
\end{align}
Selecting $\epsilon_1$  sufficiently small, say
\begin{align} \label{epscondition1}
C \epsilon_1^{1/d} \leq 1/2^{d+2+2/d},
\end{align}
One has that
$$
F(r_{n+1})\leq r_{n+1}^{d+1+\frac{2}{d}} \epsilon_1/2.
$$
Concerning the second term in $(I)_{n+1}$, we will utilize Lemma~\ref{diffpressureestimate}, set $r= r_{n+1},\, \rho = 1/4,\, n \geq 2$ there,  one can  deduce that
\begin{align} \label{3estimate4}
& Cr_{n+1}^{\frac{2}{d}}E_1 (r_{n+1})^{2/d} E_2  (3 r_{n+1}/2 )^{1/2} D (3 r_{n+1}/2 )^{1/2-1/d} \no\\
& \quad \leq Cr_{n}^{2/d} (C\epsilon_1 r_n^d)^{2/d} (C\epsilon_1 r_n^d)^{1/2} (C\epsilon_1 r_n^d)^{1/2-1/d}
  \leq C r_n^{d+1+2/d} \epsilon_1^{1+1/d}.
\end{align}
We point out that in the first inequality, Proposition~\ref{stlffestimate} is used. In addition,
\begin{align}
 &\sup_{-r_{n+1}^2<t<0} \int_{\frac{3}{2}r_{n+1}\leq |y|<1/4} |y|^{-d} |u|^2 dy  \no  \\
 &\leq \sup_{-r_{n+1}^2<t<0}\int_{\frac{3}{2}r_{n+1}\leq |y|<r_n}|y|^{-d} |u|^2 dy +
 \sum_{k=2}^{n-1} \sup_{-r_{n+1}^2<t<0}\int_{r_{k+1}\leq |y|<r_k} |y|^{-d} |u|^2 dy  \no \\
 &\leq C r_n^{-d} \sup_{-r_n^2<t<0}\int_{|y|<r_n} |u|^2 dy + C\sum_{k=2}^{n-1} r_k^{-d}\sup_{-r_k^2<t<0}\int_{|y|<r_k} |u|^2 dy \no \\
 & \leq C n\epsilon_1.
 \end{align}
Similarly,
\begin{align}
\int_{-r_{n+1}^2}^0 \int_{\frac{3}{2}r_{n+1}\leq |y|<1/4} |y|^{-d} |\nabla u|^2 dydt \leq Cn\epsilon_1.
\end{align}
Consequently,
\begin{align} \label{3esitmate5}
& C r_{n+1}^{d/2+2} E_1 (r_{n+1})^{1/2} \sup_{-r_{n+1}^2<t<0} \left(\int_{\frac{3}{2}r_{n+1}\leq|y|<1/4} |y|^{-d} |u|^2 dy \right )^{1/2}   \left( \int_{-r_{n+1}^2}^0 \int_{\frac{3}{2}r_{n+1}\leq|y|<1/4} |y|^{-d} |\nabla u|^2  dydt\right)^{1/2} \no \\
& \quad \leq C r_n^{d/2+2}(C\epsilon_1 r_n^d)^{1/2} Cn\epsilon_1
 \leq Cn r_n^{d+2} \epsilon_1^{3/2}.
\end{align}
For the last two terms in~\eqref{3estimate3}, we have
\begin{align}
C 4^{d+1} r_{n+1}^{\frac{d}{2}+3} E_1 (r_{n+1})^{\frac{1}{2}} E_1 (1/4)  \leq C r_{n}^{\frac{d}{2}+3}(C\epsilon_1r_n^{d})^{\frac{1}{2}} C \epsilon_1
 \leq C r_n^{d+3} \epsilon_1^{3/2}
\end{align}
and
\begin{align} \label{3esitmate6}
C 4^{d+1} r_{n+1}^{\frac{d}{2}+1+\frac{2}{\alpha'}}E_1 (r_{n+1})^{\frac{1}{2}}K_{\alpha}(1/4)^{\frac{1}{\alpha}} &\leq C r_n^{\frac{d}{2}+1+\frac{2}{\alpha'}} (C\epsilon_1 r_n^d)^{\frac{1}{2}} (C \epsilon_1)^{\frac{1}{\alpha}} \no \\
& \leq C r_n^{d+1+\frac{2}{\alpha'}} \epsilon_1^{1/2+1/\alpha}.
\end{align}
Noticing that $1<\alpha < 2$, we can obtain from~\eqref{3estimate4}, \eqref{3esitmate5}-\eqref{3esitmate6} that
\begin{align}
L(r_{n+1}) \leq C r_{n+1}^{d+1+ \gamma_{d,\alpha}} \epsilon_1^{1+\theta_{d,\alpha}},
\end{align}
with
\begin{align}
\gamma_{d,\alpha} = \min\Big\{ \frac{2}{d}, \frac{2}{\alpha'}\Big \}, \ \ \
\theta_{d,\alpha} = \min\Big\{\frac{1}{d}, \frac{1}{\alpha}-\frac{1}{2} \Big\}.
\end{align}
Now taking $\epsilon_1$ small enough, such that
\begin{align} \label{choiceep2}
C \epsilon_1^{\theta_{d,\alpha}} \leq 1/2.
\end{align}
So $(I)_{n+1}$ follows.

{\it Step 3}. Assuming that  $n\geq 3$, $(I)_k$ holds for $3\leq k \leq n$, we show the result of $(R)_n$. Recall the local energy inequality
\begin{align} \label{localenergyspe}
\int_{B(1)} |u(x,t)|^2 \phi_n(x,t) dx + &2\int_{-1}^t \int_{B(1)}|\nabla u|^2 \phi_n(x,\tau) dxd\tau   \no \\
&\leq  \int_{-1}^t \int_{B(1)} |u|^2(\partial_t \phi_n + \Delta \phi_n) +
u\cdot \nabla \phi_n(|u|^2 + 2p) dxd\tau
\end{align}
holds for all $t\in (-1,0)$ and $0\leq \phi_n \in C_0^{\infty}(Q(1))$. In particular, we choose $\phi_n = \chi \varphi_n$, with $\chi \in C_0^{\infty}(Q(1/3))$, $0\leq \chi \leq 1$ and $\chi= 1$ on $Q(1/4)$,
\begin{align}
\varphi_n(x,t) = \frac{1}{(r_n^2-t)^{\frac{d}{2}}} \exp{\bigg\{-\frac{|x|^2}{4(r_n^2 -t)}\bigg\}}.
\end{align}
Obviously, $\varphi_n$ differs with the  backward heat kernel by a constant and $\phi_n \geq 0$. Now one can show via a direct calculation that
\begin{itemize}
\item  $\partial_t \phi_n + \Delta \phi_n \leq C$ \ \ for all $(x,t) \in Q(1)$.

\item  $c r_n^{-d}\leq \phi_n \leq C r_n^{-d}$, \ \ \ $|\nabla \phi_n | \leq C r_n^{-(d+1)}$ \ \ on\ \ $Q(r_n)$, \ \ $n \geq 2$.

\item $\phi_n \leq C r_k^{-d}$, \ \ \ $|\nabla \phi_n| \leq C r_k^{-(d+1)}$ \ \ \ on \ \  $Q(r_{k-1})\backslash Q(r_k)$, \ \  $1<k \leq n$.

\end{itemize}
for some constant $c,\ C$ depending only on $d$. It follows from~\eqref{localenergyspe} that
\begin{align}
\sup_{-r_n^2 <t<0}\fint_{B(r_n)} |u(x,t)|^2 dx + r_n^{-d} \int_{Q(r_n)} |\nabla u|^2 dxd\tau  \leq C (I + II + III),
\end{align}
where
\begin{align}
&\qquad \qquad  I = \int_{Q(1)} |u|^2 (\partial_t \phi_n + \Delta \phi_n) dxd\tau, \no \\
& II =  \int_{Q(1)} |u|^2 (u\cdot \nabla \phi_n)  dxd\tau , \ \quad   III= \int_{Q(1)} p (u\cdot \nabla \phi_n)  dxd\tau.  \no
\end{align}
Thus we are reduced to discuss the above three terms, one can readily get
\begin{align}
I \leq C\int_{Q(r_1)} |u|^2 dxdt \leq C \epsilon_1.
\end{align}
The estimate of $II$ and $III$ is a bit complicated, nevertheless goes in a similar way, both fully exploit the divergence free condition of the solution $u$.  Let $\eta_k,\, k=1,\ldots,n $ be smooth cut-off functions, satisfying
\begin{align}
\textrm{supp}\,\eta_k \subset Q(r_k), \ \ \ \ 0 \leq \eta_k  \leq 1,  \no \\
\eta_k =1 \ \ \textrm{on} \ \ Q(7r_k/8), \ \ \ \  |\nabla \eta_k|\leq cr_k^{-1}.
\end{align}
By a direct computation, one can  see
$$\eta_1 \phi_n = \phi_n, \ \ \ |\nabla ((\eta_{k-1}-\eta_k) \phi_n)| \leq C r_{k-1}^{-(d+1)}, \ \ k=2,\ldots,n.  $$
 Therefore
\begin{align}
II&= \int_{Q(r_1)} |u|^2  u\cdot \nabla (\eta_1 \phi_n) dxdt \no  \\
&= \sum_{k=2}^n \int_{Q(r_{k-1})} |u|^2 u\cdot \nabla ((\eta_{k-1}-\eta_k) \phi_n)  dxdt + \int_{Q(r_n)} |u|^2 u\cdot \nabla (\eta_n \phi_n)  dxdt. \no
\end{align}
By means of the argument that results in~\eqref{cubicestimate1}, one can show
\begin{align}
\int_{Q(r_{k-1})} \left||u|^2-[|u|^2]_{B(r_{k-1})}\right| |u| dxdt \leq C r^{1+d+2/d}_{k-1} \epsilon_1^{1+1/d}, \ \ \ k=2,\,3.
\end{align}
Due to the fact that  ${\rm{div}}\,u = 0$ and the hypothesis in Step $III$,  we can assert that for $2 \leq k \leq  n$,
\begin{align}
\int_{Q(r_{k-1})} |u|^2 u\cdot \nabla ((\eta_{k-1}-\eta_k) \phi_n)  dxdt &= \int_{Q(r_{k-1})} (|u|^2-[|u|^2]_{B(r_{k-1})}) u\cdot \nabla ((\eta_{k-1}-\eta_k) \phi_n)  dxdt  \no  \\
&\leq C \frac{1}{r_{k-1}^{d+1}}\int_{Q(r_{k-1})} \left||u|^2-[|u|^2]_{B(r_{k-1})} \right| |u| dxdt  \no \\
& \leq C r_{k-1}^{2/d} \epsilon_1.
\end{align}
Similarly
\begin{align}
\int_{Q(r_n)} |u|^2 u\cdot \nabla (\eta_n \phi_n)  dxdt \leq C r_n^{2/d} \epsilon_1.
\end{align}
This implies
\begin{align}
II \leq C \sum_{k=2}^n r_{k-1}^{2/d} \epsilon_1 + C r_n^{2/d} \epsilon_1  \leq C \epsilon_1.
\end{align}
Finally, we treat $III$, as before,
\begin{align}
III= \sum_{k=2}^n \int_{Q(r_{k-1})} p  u\cdot \nabla ((\eta_{k-1}-\eta_k) \phi_n)  dxd\tau + \int_{Q(r_n)} p u\cdot \nabla (\eta_n \phi_n) dxd\tau. \no
\end{align}
When $k = 2,3$, it follows from~\eqref{smallconst} that
\begin{align}
\int_{Q(r_{k-1})} p  u\cdot \nabla ((\eta_{k-1}-\eta_k) \phi_n)  dxd\tau \leq C \int_{Q(1)} |u||p| dxdt \leq C \epsilon_1.
\end{align}
While for $4 \leq k \leq n$, we have
\begin{align}
\int_{Q(r_{k-1})} p  u\cdot \nabla ((\eta_{k-1}-\eta_k) \phi_n)  dxd\tau & = \int_{Q(r_{k-1})} (p-[p]_{B(r_{k-1})})u\cdot \nabla ((\eta_{k-1}-\eta_k) \phi_n)  dxd\tau \no \\
& \leq C \frac{1}{r_{k-1}^{d+1}} \int_{Q(r_{k-1})}\left|p-[p]_{B(r_{k-1})}\right||u| dxdt  \no \\
& \leq C r_{k-1}^{\gamma_{d,\alpha}} \epsilon_1.
\end{align}
In the same way,
\begin{align}
\int_{Q(r_n)} p u\cdot \nabla (\eta_n \phi_n) dxd\tau  \leq C r_n^{\gamma_{d,\alpha}} \epsilon_1.
\end{align}
So one can find
\begin{align}
III  \leq C\epsilon_1 + \sum_{k=4}^n r_{k-1}^{\gamma_{d,\alpha}}\epsilon_1 + C r_n^{\gamma_{d,\alpha}} \epsilon_1 \leq C \epsilon_1.
\end{align}
Gathering the estimates of $I,\, II$ and $III$, we finally obtain that $(R)_n$ holds, which is exactly the required result.  The proof is done.
\end{proof}

We mention a bit more on the choice of $\epsilon_1$. By a closer observation, one can figure out that various constants $C$ appearing in the course of  Step $III$  relies on $d$, $M$ and $\alpha$ only. The same applies for the  constants $C$ in~\eqref{epscondition1} and~\eqref{choiceep2},  we can specify $\epsilon_1$ through~\eqref{epscondition1} and~\eqref{choiceep2}. 

To show  Corollary~\ref{modifiedversionregu},  we need to control the local energy in terms of $\int_{Q(3r/4)}|u|^3dxdt$ and $\int_{Q(3r/4)} |p|^{3/2} dxdt$. First, assume $u \in L^{\infty}(-4,0; \dot B^{-1}_{\infty,\infty})$ and  $u, p$ is a pair of  suitable weak solution on $Q(2)$, applying the local energy inequality with  test function $\phi$ satisfying ${\rm supp}\,\phi \subset Q(3/2)$ and $\phi = 1$ on $Q(1)$, we can find
\begin{align} \label{controllocal}
& \sup_{-1<t<0}\int_{B(1)}|u|^2 dx + \int_{Q(1)} |\nabla u|^2 dxdt  \no \\
&\ \ \leq C \bigg( \int_{Q(3/2)}|u|^2 +|u|^3 + |u||p| dxdt \bigg) \no \\
& \ \ \leq C\bigg( \int_{Q(3/2)}|u|^3 dxdt \bigg )^{2/3} + C\bigg( \int_{Q(3/2)} |u|^3 + |p|^{3/2} dxdt \bigg).
\end{align}
In general, for $u,\,p$ defined on $Q(r)$, we set
\begin{align}
u_r(x,t) &= \lambda u\big(\lambda x, \lambda^2 t\big), \ \ \ \ \lambda= r/2.   \no  \\
p_r(x,t) &= \lambda^2 p\big(\lambda x, \lambda^2 t \big).
\end{align}
As such, $u_r,\, p_r$ become functions defined on $Q(2)$.
\begin{proof}[Proof of Corollary~\ref{modifiedversionregu}]
Let $u_r,\,p_r$ be as above, by~\eqref{controllocal} and H\"older inequality, we have
\begin{align}
& \sup_{-1<t<0}\int_{B(1)} |u_r|^2 dx + \int_{Q(1)}(|\nabla u_r|^2+|u_r||p_r|) dxdt + \int_{-1}^0 \bigg(\int_{B(1)}|p_r|dx \bigg)^{3/2} dt  \no  \\
& \leq C \bigg( \int_{Q(3/2)}|u_r|^3 dxdt \bigg )^{2/3} + C\bigg( \int_{Q(3/2)} |u_r|^3 + |p_r|^{3/2} dxdt \bigg) \no \\
& \leq C \bigg( \frac{1}{r^{d-1}}\int_{Q(3r/4)}|u|^3 dxdt \bigg )^{2/3} + C\bigg( \frac{1}{r^{d-1}}\int_{Q(3r/4)}|u|^3  + |p|^{3/2} dxdt \bigg) \no \\
& \leq C (\tilde{\epsilon_1})^{3/2} + C \tilde{\epsilon_1}.
\end{align}
Selecting $\tilde{\epsilon_1}$ small enough, such that
\begin{align}
C (\tilde{\epsilon_1})^{3/2}+ C \tilde{\epsilon_1} \leq \epsilon_1.
\end{align}
Here $\epsilon_1$ is given by~\eqref{smallhypo}  with $\alpha = 3/2$. 
Then Theorem~\ref{regularitycriterion} can infer
\begin{align}
\sup_{(x,t)\in Q(1/2)} |u_r(x,t)| \leq C \epsilon_1.
\end{align}
This concludes the proof.
\end{proof}


\section{Mild solution in critical Besov space}  \label{proofregu}
This section is devoted to proving Theorem~\ref{mildblowupbesov}, and we  argue by contradiction. Assume that the conclusion of Theorem~\ref{mildblowupbesov} does not hold, i.e., there exists $M>0$ such that
\begin{align} \label{upperbd}
  \|u(t)\|_{\dot B^{s_p}_{p,q}(\R^d)}  \leq M, \ \ a.e. \ t\in [0,T_*).
\end{align}
\subsection{Formation of singular point at blowup time}
We shall show the existence of  singular point for blowup mild solution in critical Besov space under an extra regularity assumption, the key part lies in establishing some global space-time bounds for the solution until the singular time. Let us first recall the local Cauchy theory for NS with initial data in $\dot B^{s_p}_{p,q}$, see~\cite{albritton16} for the 3D case and the higher dimensional cases are similar.
\begin{thm} \label{subcribesov}
Let $u_0 \in \dot B^{s}_{p, q}(\R^d)$ with $d<p,\,q < \infty$, $s_p \leq s <0$,  Then there exist a time $T>0$ and a unique mild solution
$ u:= NS(u_0) \in C([0,T]; \dot B^s_{p,q}) \cap \mathscr{K}^s_{p,\infty}(T)$ to~\eqref{ns}, such that
\begin{align} \label{localexistencebdd}
  \|\partial^l_t \nabla^j u\|_{\mathscr{K}^{s-2l-j}_{p,\infty}(T)}+ \|u\|_{L^{\infty}(0,T;\dot B^{s}_{p,q})} \leq C \|u_0\|_{\dot B^s_{p,q}}, \ \ \ l,\,j \in \{0,1\}.
\end{align}
Moreover, we can take $T \geq c_0 \|u_0\|^{2/(s-s_p)}_{\dot B^s_{p,q}}$, provided $s>s_p$, here $c_0$ is independent of $u_0$. Recall that
$$
\|u\|_{\mathscr{K}^{\beta}_{p,\infty}(T)}:= \sup_{0<t<T} t^{-\beta/2}\|u(x,t)\|_{L_x^p}, \ \ \beta<0.
$$
\end{thm}

We further exploit a regularity result for the mild solution with data in $\dot B^{s}_{p,q}$, $s_p\leq s<0$.
\begin{prop} \label{linftybd}
Let $u$ be the mild solution given by Theorem~\ref{subcribesov} and the estimate~\eqref{localexistencebdd} hold. Additionally, assume $2d<p<\infty$, then
\begin{align} \label{ptlinfty}
\|u\|_{L^{\infty}(\sigma,T; L^{\infty})} \leq C(\sigma, T, \|u_0\|_{\dot B^{s}_{p,q}}), \ \ \ \forall\, \sigma \in (0,T).
\end{align}
Particularly, $u \in C^{\infty}((0,T)\times \R^d)$.
\end{prop}
\begin{proof}
For simlicity, we denote $\delta = s-s_p$ and obviously, $\delta \in [0,1)$.  It is known that $u$ can be written as
\begin{align}
u(t) & =   e^{ t \Delta } u_0 -  \int^t_0  e^{ (t-\tau) \Delta } \mathbb{P}\ \textrm{div}(u\otimes u)(\tau) d\tau \no  \\
     & = u_{L} - B(u,u).
\end{align}
The estimate of the linear term $u_{L}:=  e^{ t \Delta } u_0 $ follows from~\eqref{besovheat}, since
\begin{align} \label{linearinfty}
\sup_{0<t<T} t^{(1-\delta)/2} \|u_{L}\|_{L^{\infty}} \leq C \|u_0\|_{\dot B^{-1+\delta}_{\infty,\infty}} \leq C \|u_0\|_{\dot B^{s}_{p,q}}.
\end{align}
On the other hand, by~\cite{pglr02}, the bilinear term $B(u,u)$ can be formulated as
\begin{align}
B(u,u)(t) = \int_0^t \frac{1}{(t-\tau)^{\frac{d+1}{2}}} G\Big(\frac{x}{\sqrt{t-\tau}} \Big) * (u\otimes u )(\tau) d\tau,
\end{align}
where $G(x)$ satisfies
$$
|G(x)|\leq \frac{C}{(1+|x|)^{d+1}}.
$$
Let $r$ be such that $1= 2/p +1/r$, $t\in (\sigma, T)$,   applying Young inequality, one can figure out that
\begin{align}
\|B(u,u)\|_{L^{\infty}} &\leq C \int_0^t \frac{1}{(t-\tau)^{\frac{d+1}{2}}} \Big\|G\Big(\frac{x}{\sqrt{t-\tau}} \Big)\Big\|_{L^{r}} \|u(\tau)\|^2_{L_x^p} d\tau  \no \\
& \leq C \|u\|^2_{\mathscr{K}^s_{p,\infty}(T)} \int_0^t \frac{\tau^{s}}{(t-\tau)^{1/2 +d/p}} d\tau  \no \\
& \leq C t^{\delta -1/2} \|u\|^2_{\mathscr{K}^s_{p,\infty}(T)}  \int_0^1 \frac{\lambda^s}{(1-\lambda)^{1/2 +d/p}} d \lambda. \no
\end{align}
Note that $-1<s<0$,\ $2d < p < \infty$ and~\eqref{localexistencebdd}, one can readily see
\begin{align}
\|B(u,u)\|_{L^{\infty}(\sigma,T; L^{\infty})} \leq C(\sigma, \delta, T, \|u_0\|_{\dot B^s_{p,q}}).
\end{align}
This combining with~\eqref{linearinfty} yields the desired bound. Once~\eqref{ptlinfty} is established, the smoothness becomes an immediate result, see~\cite{pglr02}.
\end{proof}

Next result is related to the decomposition of functions in Besov space, which can be viewed from the point of interpolation theory, here we present a simple version, see~\cite{albritton16} for the proof. As for the slightly general case, one can refer to~\cite{barker16}.
\begin{lem} \label{besovdec}
Let $d<p<m< \infty$ and $\theta \in (0,1)$ be such that
\begin{align}
\frac{1}{p} = \frac{\theta}{2} + \frac{1-\theta}{m}.
\end{align}
Define $s$ by $s_p = (1-\theta)s$.  Given $\eta >0$ and a vector field $v \in \dot B^{s_p}_{p,p}(\R^d)$, there exist vector fields $U\in  \dot B^{s_p}_{p,p} \cap L^2$ and $V \in \dot B^{s_p}_{p,p}\cap \dot B^{s}_{m,m}$, verifying  $v= U+V$ and
\begin{align}
\|U\|^2_{L^2} &\leq  C\eta^{2-p} \|v\|^p_{\dot B^{s_p}_{p,p}}, \\
\|V\|^m_{\dot B^s_{m,m}} &\leq  \eta^{m-p} \|v\|^p_{\dot B^{s_p}_{p,p}}, \\
\|U\|_{\dot B^{s_p}_{p,p}} + \|V\|_{\dot B^{s_p}_{p,p}} &\leq C \|v\|_{\dot B^{s_p}_{p,p}}.
\end{align}
Further, $U$ and $V$ can be selected to be divergence free provided that $v$ is divergence free.
\end{lem}
When  making standard  energy estimate for NS equation, we need to deal with some type of trilinear form, specifically, the integral $\int_0^T \int_{\R^d} v\otimes u:\nabla u dxdt $ with $u \in E_T$, $v$ has some sort of regularity condition, here
\begin{align} \label{energyspace}
E_T:= L^{\infty}(0,T; L^2(\R^d)) \cap L^2(0,T;\dot H^1(\R^d)).
\end{align}
The following result gives a proper estimate of that kind, and is adapted to our needs later. One can refer to~\cite{Gi86,barker16} and references therein for the proof.
\begin{lem}\label{trilinearesti}
Let $d\geq 3$, $u\in E_T$, $v\in L^r(0,T;L^q(\R^d))$ with $2/r+ d/q = 1,\,d<q<\infty$. Then a constant $C$ exists, such that
 \begin{align}
 \|v \otimes u\|_{L^2(0,T;L^2)}  \leq C \|v\|_{L^r(0,T;L^q)} \|u\|^{1-\theta}_{L^{\infty}(0,T;L^2)} \|u\|^{\theta}_{L^{2}(0,T;\dot H^1)},\ \ \
 \theta = d/q.
 \end{align}
 Moreover, for any $\epsilon>0$, there exists a constant $C_{\epsilon}$, such that
\begin{align}
\int_0^T \int_{\R^d} v\otimes u:\nabla u dxdt \leq \epsilon \int_0^{T} \|\nabla u\|^2_{L^2} dt +
C_{\epsilon} \int_0^{T} \|v\|^r_{L^q} \|u\|^2_{L^2}dt.
\end{align}
\end{lem}
We now state our main result in this part. For simplicity, we restrict ourselves to consider the initial data in  $\dot B^{s_p}_{p,p}$, which is enough for our later purpose.
\begin{prop} \label{bdatinf}
Let $4\leq d<p<\infty,\, u_0\in \dot B^{s_p}_{p,p}$ and $u:= NS(u_0)$ be the  mild solution to~\eqref{ns}, assume further $u\in L^{\infty}(0,T_*; \dot B^{s_p}_{p,p})$   blows up at finite time $T_*$,  then there exists some $R>0$, such that
\begin{align}
\sup_{T_*/2 <t< T_*} \sup_{x\in B(R)^c} |u(x,t)| < \infty.
\end{align}
\end{prop}
\begin{proof}
Let $p<m<\infty$ satisfy
\begin{align}
\frac{1}{p} = \frac{1-\theta}{m} + \frac{\theta}{2}, \ \, s_p = (s_m + \delta)(1-\theta).
\end{align}
A simple computation shows $\delta = (d-2)\theta/[2(1-\theta)]$.   According to Lemma~\ref{besovdec},
we can decompose $u_0$ as $u_0 = u_{0,1} + u_{0,2}$ with $u_{0,1} \in \dot B^{s_m+ \delta}_{m,m} \cap \dot B^{s_p}_{p,p}$,\ \ $u_{0,2} \in L^2 \cap \dot B^{s_p}_{p,p}$. Besides,
\begin{align}
\|u_{0,1}\|^m_{\dot B^{s_m+\delta}_{m,m}} \leq& \eta^{m-p} \|u_0\|^p_{\dot B^{s_p}_{p,p}},\ \ \ \|u_{0,2}\|^2_{L^2} \leq C\eta^{2-p} \|u_0\|^p_{\dot B^{s_p}_{p,p}}, \label{4estimate1} \\
&\|u_{0,1}\|_{\dot B^{s_p}_{p,p}} + \|u_{0,2}\|_{\dot B^{s_p}_{p,p}} \leq C \|u_0\|_{\dot B^{s_p}_{p,p}}. \label{4estimate2}
\end{align}
Define $V = NS(u_{0,1})$ and $U = u- V$. Taking $\eta $ to be small enough,   we see the existence time $T$ of solution $V$ given by Theorem~\ref{subcribesov} can be beyond $T_*$, so
\begin{align}
\|V\|_{\mathscr{K}^{s_m+\delta}_{m,\infty}(T_*)} + \|\nabla V\|_{\mathscr{K}^{s_m+\delta-1}_{m,\infty}(T_*)} &\leq C \|u_{0,1}\|_{\dot B^{s_m+\delta}_{m,m}} \leq C(\eta, \|u_0\|_{\dot B^{s_p}_{p,p}}), \no \\
\|V\|_{L^{\infty}(0,T_*; \dot B^{s_p}_{p,p})} &\leq C(\eta, \|u_0\|_{\dot B^{s_p}_{p,p}}),
\end{align}
In addition, determining $r$ by $2/r + d/m =1$, one can verify
\begin{align}
\|V\|^r_{L^r(0,T_*; L^m)} \leq \|V\|^r_{\mathscr{K}^{s_m+\delta}_{m,\infty}(T_*)} \int_0^{T_*} t^{r/2(s_m+\delta)} dt \leq C T_*^{r\delta/2}.
\end{align}
where $C$ depends on $r,\,\delta,\, \eta$ and $\|u_0\|_{\dot B^{s_p}_{p,p}}$. Let $Q_1$ be the associated pressure with $V$, then
 \begin{align}
 \Delta Q_1 = -{\rm div}\,(V\cdot \nabla V).
 \end{align}
 It follows from the classical Calderon-Zygmund estimate that
\begin{align}
\sup_{T_*/4<t<T_*} \|Q_1\|_{L^{m/2}} \leq C\big(\eta, T_*, \|u_0\|_{\dot B^{s_p}_{p,p}}\big).
\end{align}
On the other hand, $U$ solves the following perturbative Navier-Stokes equation:
\begin{equation*}
\begin{cases}
\partial_t U - \Delta U + U \cdot\nabla U + V \cdot \nabla U + U \cdot \nabla V + \nabla Q_2 =0,  & \\
{\rm div}\, U =0, & \\
U(0,x) = u_{0,2}(x), &
\end{cases}
\end{equation*}
where  $Q_2$ satisfies
\begin{align}
\Delta Q_2 = -{\rm div}(U\cdot \nabla U + U\cdot \nabla V + V\cdot \nabla U ).
\end{align}
Noticing that
 $$ u,\,V \in L^{\infty}(0,T_*;\dot B^{s_p}_{p,p}),$$
so does $U$.  At the same time, as  $u_{0,2} \in L^2$, by persistence and propagation of regularity (cf. \cite{GaIfPl03}), one can further show there exists a time $T_1 \leq T_*$, such that
\begin{align}
U \in L^{\infty}(0,T_1;L^2) \cap L^2(0,T_1;\dot H^1).
\end{align}
Hence, performing the standard energy estimate, we can see that
\begin{align}
\|U(t)\|^2_{L^2} + 2 \int_{0}^t \|\nabla U(\tau)\|^2_{L^2} d\tau = \|u_{0,2}\|^2_{L^2} + 2 \int_0^t \int_{\R^d} V \otimes U: \nabla U dx d\tau
\end{align}
fulfills for all $t\in (0, T_1)$.  This together with Lemma~\ref{trilinearesti} and Gronwall inequality yields
\begin{align}
\|U(t)\|^2_{L^2} +  \int_{0}^t \|\nabla U(\tau)\|^2_{L^2} d\tau  \leq \|u_{0,2}\|^2_{L^2} \exp{\bigg(C\int_{0}^{T_*} \|V(\tau)\|^r_{L^m} d\tau \bigg)},\ \ \ 0\leq t < T_1.
\end{align}
However, the above boundedness of $U$ in the energy space can ensure that $T_1 = T_*$, so $U \in E_{T_*}$. Noting that
\begin{align}
U\in  L^{\infty}(0,T_*;\dot B^{s_p}_{p,p}).
\end{align}
 By Proposition~\ref{lffestimate}, we have  $U \in L^4(0,T_*;L^4)$.
On account of Lemma~\ref{trilinearesti}, we see that $U\otimes V \in L^2(0,T_*;L^2)$. Now applying  Calderon-Zygmund estimate, one can obtain that $Q_2 \in L^2(0,T_*;L^2)$. To summary, $u = V+ U$, with
\begin{align} \label{decomglobalbound}
V \in L^{\infty}(T_*/4,T_*; L^m(\R^d)), \ \ \ U \in L^4(0,T_*; L^4(\R^d)).
\end{align}
The associated pressure $q = Q_1 + Q_2$, with
\begin{align} \label{pressureglobal}
Q_1 \in L^{\infty}(T_*/4,T_*; L^{m/2}(\R^d)), \ \ \ Q_2 \in L^2(0,T_*; L^2(\R^d)).
\end{align}
Moreover, $u,\,q$ forms a pair of suitable weak solution on any bounded domain of $(T_*/4,T_*) \times \R^d $.  Due to~\eqref{decomglobalbound} and~\eqref{pressureglobal}, one can claim that for any $\epsilon, \rho>0$,  there exists some $R_0>0$, such that
\begin{align}
\sup_{z\in \mathbb{R}^d} \int_{T_*/4}^{T_*} \int_{B(R_0)^c \cap B(z,\rho)}|u|^3 + |q|^{3/2}dxdt  < \epsilon/T_*^{{(d-1)/2}}.
\end{align}
Fix a point $z_0=(x_0,t_0)$, $T_*/2<t_0<T_*$. Choosing $\rho= \sqrt{T_*}/2$, one sees $Q(z_0,\rho) \subset (T_*/4,T_*) \times \R^d $ and
\begin{align}
\|u\|_{L^{\infty}(t_0-\rho^2,t_0; \dot B^{-1}_{\infty,\infty})} \leq \|u\|_{L^{\infty}(0,T_*; \dot B^{-1}_{\infty,\infty})}.
\end{align}
The following integral
\begin{align}
\frac{1}{\rho^{d-1}}\int_{Q(z_0,3\rho/4)} |u|^3 + |q|^{3/2} dxdt   < C \epsilon
\end{align}
provided  $|x_0| > R_0+ \rho $. Now we take $\epsilon$ so small that $C\epsilon < \tilde{\epsilon}_1$.  Therefore, Corollary~\ref{modifiedversionregu} implies the boundedness of $u$ around $z_0$. The proof is completed.
\end{proof}

\subsection{Some priori estimates and limiting process} \label{prioriestimate}

This subsection presents some preparation results for the  proof of Theorem~\ref{mildblowupbesov}. Since $\dot B^{s_p}_{p,q} \subset \dot B^{s_r}_{r,r}$ for $r = \max{\{p,q\}}$, it suffices to prove the theorem with initial data $ u_0 \in \dot B^{s_p}_{p,p}, \  4\leq d<p<\infty$, see~\cite{GaIfPl03,GKP} for further explanations. Taking the assumption of Proposition~\ref{linftybd} into consideration, we will assume, from now on that in Theorem~\ref{mildblowupbesov}, the initial data $u_0$ fulfills
\begin{align} \label{initialhypo}
u_0 \in \dot B^{s_p}_{p,p}, \ \ \ 2d < p< \infty.
\end{align}

Let $u:= NS(u_0)$ be the mild solution described in Theorem~\ref{mildblowupbesov}, assume the conclusion there is false, by \eqref{upperbd}, there exists $\widetilde{M}>0$ such that
\begin{align} \label{bddbesov}
\|u\|_{L^{\infty}(0,T_*;\dot B^{s_p}_{p,p})} \leq \widetilde{M}.
\end{align}
As $u$ becomes singular at $T_*$, so
\begin{align}
\lim\sup_{t\to T_*} \|u(t)\|_{L^{\infty}(\R^d)} = \infty.
\end{align}
However, Proposition~\ref{bdatinf} implies the boundedness of $u$ out of $B(R)$ for some $R>0$, it follows that there exists some point $Z_0:=(X_0, T_*)$, such that $u$ is singular at $Z_0$,  more precisely,
\begin{align}
u \notin L^{\infty}(Q(Z_0,r)), \ \ \ \forall\, 0<r<\sqrt{T_*}.
\end{align}

Let $q$ be the pressure associated with $u$,  we plan to rescale  $u,\,q$ around $Z_0$, then derive a solution sequence. First, by~\eqref{bddbesov}, one can find a time sequence $\{t_n\}_{n\geq 1}$, such that $t_n \to T_*$ as $n \to \infty$, and
\begin{align}
\|u(t_n)\|_{\dot B^{s_p}_{p,p}} \leq \|u\|_{L^{\infty}(0,T_*;\dot B^{s_p}_{p,p})}\leq \widetilde{M}.
\end{align}
Without loss of generality, one can  assume
\begin{align}
u(t_n) \rightharpoonup  u^*:= u(T_*) \ \ \ \textrm{weakly in} \ \ \ \dot B^{s_p}_{p,p}.
\end{align}
Set $\lambda_n = \sqrt{(T_* -t_n)/2}$ and denote
\begin{align}
v_n(x,t) &= \lambda_n u(X_0 + \lambda_n x, T_*+ \lambda^2_n t), \\
q_n(x,t) &= \lambda_n^2q(X_0 + \lambda_n x, T_*+ \lambda^2_n t).
\end{align}
Naturally, $v_n $ is a mild solution to~\eqref{ns} on $(-2,0) \times \R^d$ with initial data
$$v_n(x,-2)= v_{0,n}:= \lambda_n u(X_0+ \lambda_n x, t_n).$$
By a direct calculation, one can see
\begin{align}
\|v_n\|_{L^{\infty}(-2,0; \dot B^{s_p}_{p,p})}  = \|u\|_{L^{\infty}(t_n,T_*; \dot B^{s_p}_{p,p})} &\leq \widetilde{M},  \label{bdbesov}\\
\|v_{0,n}\|_{\dot B^{s_p}_{p,p}}  = \|u(t_n)\|_{\dot B^{s_p}_{p,p}} &\leq \widetilde{M}.
\end{align}

Next, we aim at obtaining some uniform control over $v_n,\,q_n$, the procedure is quite similar to the proof of Proposition~\ref{bdatinf} and we will omit the details of the argument by simply writing down relevant conclusions. For convenience, the notations $ m,\,\delta,\, r $ used in the proof of Proposition~\ref{bdatinf} will be continuously used. First, there exist $v^1_{0,n} \in \dot B^{s_m +\delta}_{m,m} \cap \dot B^{s_p}_{p,p}$
and $v_{0,n}^2 \in L^2 \cap \dot B^{s_p}_{p,p} $, such that $ v_{0,n} = v^1_{0,n} + v^2_{0,n}$ and the corresponding qualitative estimates   hold, i.e. \eqref{4estimate1} and~\eqref{4estimate2} with $u_{0,1},\,u_{0,2}$ replaced by $v^1_{0,n},\,v^2_{0,n}$ respectively.

Given the decomposition of the initial data, we can also express the solution into two parts, set
 \begin{align}
  U_n^1:= NS(v_{0,n}^1), \ \ \ U_n^2 = v_n - U_n^1.
 \end{align}
 Let us treat $U_n^1$ now, by choosing $\eta$ to be sufficiently small and applying local Cauchy theory of NS, see Theorem~\ref{subcribesov}, one can deduce that
\begin{align}
\|U_n^1\|_{\mathscr{K}^{s_m+\delta}_{m,\infty}(-2,0)} + \|\nabla U_n^1\|_{\mathscr{K}^{s_m+\delta-1}_{m,\infty}(-2,0)} & \leq C \|v_{0,n}^1\|_{\dot B^{s_m+\delta}_{m,m}}  \leq C(\eta, \widetilde{M}).  \\
 \|\partial_t U_n^1\|_{\mathscr{K}^{s_m+\delta-2}_{m,\infty}(-2,0)} & \leq C \|v_{0,n}^1\|_{\dot B^{s_m+\delta}_{m,m}}  \leq C(\eta, \widetilde{M})
\end{align}
Meantime,
\begin{align}
\|U_n^1\|_{L^{\infty}(-2,0; \dot B^{s_p}_{p,p})}  \leq C(\eta, \widetilde{M}). \label{subbdbesov}
\end{align}
For arbitrary $0<\sigma< 2$,
\begin{align} \label{4esitmate5}
\|U_n^1\|_{L^{\infty}(-2+\sigma,\,0; L^m)} + \|\nabla U_n^1\|_{L^{\infty}(-2+\sigma,\,0; L^m)} + \|\partial_t U_n^1\|_{L^{\infty}(-2+\sigma,\,0; L^m)} \leq C(\sigma, \eta, \widetilde{M}).
\end{align}
Furthermore, Proposition~\ref{linftybd} yields
\begin{align}
\|U_n^1\|_{L^{\infty}(-2+\sigma,\,0; L^{\infty})  } \leq C(\sigma, \eta, \widetilde{M}), \ \ \forall\, 0<\sigma<2. \label{4estimate3}
\end{align}
Let $Q_n^1$ be the pressure associated with $U_n^1$, then  Calderon-Zygmund estimate infers
\begin{align} \label{4estimate6}
\|Q_n^1\|_{L^{\infty}(-2+\sigma,\,0; L^{m/2})} \leq C(\sigma, \eta, \widetilde{M}).
\end{align}
Based on the above estimates over $U_n^1,\,Q_n^1$, we can show the following result.
 \begin{lem} \label{limitoffirst}
 There exist limit functions $u_{\infty}^1,\,q_{\infty}^1$, defined on $(-2,0)\times \R^d$, such that for any $0<\sigma<2,\,R>0$,  we have
\begin{itemize}
  \item[{\rm (1)}] $U_n^1 \rightharpoonup u_{\infty}^1$\ \  weakly* in \ $L^{\infty}(-2+\sigma,\,0; L^{m})$ \ and \ $L^{\infty}(-2+\sigma,\,0; L^{\infty})$;

  \item[{\rm (2)}] $\nabla U_n^1 \rightharpoonup \nabla u_{\infty}^1$\ \  weakly* in \ $L^{\infty}(-2+\sigma,\,0; L^{m})$;

  \item[{\rm (3)}] $ U_n^1 \rightarrow  u_{\infty}^1$ \ \  strongly in \ $C([-2+\sigma,\,0], L^{m}(B(R)))$;

  \item[{\rm (4)}] $U_n^1 \rightharpoonup u_{\infty}^1$\ \ weakly* in \ $L^{\infty}(-2,0; \dot B^{s_p}_{p,p})$;

  \item[{\rm (5)}] $Q_n^1 \rightharpoonup q_{\infty}^1$\ \ weakly* in \  $L^{\infty}(-2+\sigma,\,0; L^{m/2})$.
\end{itemize}
\end{lem}
\begin{proof}
Obviously, {\rm (1), (2), (3)} and {\rm (5)} follow directly from~\eqref{4estimate3}-\eqref{4estimate6}.  For any fixed $r>1$, note that
\begin{align}
U_n^1 \in L^r(-2+\sigma,\, 0; W^1_{m}(B(R))),\, \, \  \partial_t U_n^1 \in L^r(-2+\sigma,\, 0; L^m(B(R))).
\end{align}
Thus {\rm (3)} is a consequence of Aubin-Lions lemma (cf. \cite{Si87,Sere15}).
\end{proof}

We turn to the estimate of $U_n^2$. Observing that $U_n^2$ solves the following perturbed Navier-Stokes equation on domain $(-2,0)\times \R^d$,
\begin{equation*}
\begin{cases}
\partial_t U_n^2 - \Delta U_n^2 + U_n^2 \cdot\nabla U_n^2 + U_n^2 \cdot \nabla U_n^1 + U_n^1 \cdot \nabla U_n^2 + \nabla Q_n^2 =0,  & \\
{\rm div}\, U_n^2 =0, & \\
U_n^2(x,-2) = v^2_{0,n}(x). &
\end{cases}
\end{equation*}
Due to~\eqref{bdbesov} and~\eqref{subbdbesov}, one can easily find
\begin{align} \label{uniforbesov}
U_n^2 \in L^{\infty}(-2,0; \dot B^{s_p}_{p,p}(\R^d)).
\end{align}
Recall that $r$ is such that $2/r+ d/m =1$, we  can apply energy estimate again to see
\begin{align} \label{globalenergy1}
\|U_n^2(t)\|^2_{L^2} + \int_{-2}^t \|\nabla U^2_n\|^2_{L^2} ds \leq \|v^2_{0,n}\|^2_{L^2} \exp\left(C\int_{-2}^0 \|U_n^1\|_{L^m}^r ds\right) \leq C(\eta, \widetilde{M})
\end{align}
holds for all $t\in (-2,0)$. It follows from interpolating~\eqref{uniforbesov} and~\eqref{globalenergy1} that
\begin{align} \label{4estimate8}
\|U_n^2\|_{L^4(-2,0;L^4)} \leq C(\eta, \widetilde{M}).
\end{align}
The pressure $Q_n^2$ meets
\begin{align} \label{4estimate7}
\|Q_n^2\|_{L^2(-2,0;L^2)} \leq C(\eta, \widetilde{M}).
\end{align}
For more details on the above estimates of $U_n^2,\,Q_n^2$, one can refer to the proof of Proposition~\ref{bdatinf}.  The estimate of $\partial_t U_n^2$  can be done as follows: let $R>0$, $\phi \in C_0^{\infty}(B(R))$, then
\begin{align}
|\langle\partial_t U_n^2, \phi\rangle| &= \big|-\langle\nabla U_n^2, \nabla \phi \rangle + \langle U_n^2, U_n^2\cdot \nabla \phi \rangle + \langle U_n^1, U_n^2\cdot \nabla \phi\rangle  + \langle U_n^2, U_n^1\cdot \nabla \phi\rangle + \langle Q_n^2, {\rm div} \phi \rangle \big| \no  \\
& \leq C\big(\|\nabla U_n^2 \|_{L^2} + \|U^2_n\|^2_{L^4} + \|U_n^1 \otimes U_n^2\|_{L^2} + \|Q_n^2\|_{L^2} \big) \|\nabla \phi \|_{L^2}
\end{align}
Taking $L^2$ integral with respect to time over interval $[-2,0]$ and using Lemma~\ref{trilinearesti}, one  sees
\begin{align}
\|\partial_t U_n^2\|_{L^2(-2,0; H^{-1}(B(R)))} \leq C(\eta, \widetilde{M}), \ \ \ \forall\, R>0.
\end{align}
where $H^{-1}(B(R))$ is the dual space of $H_0^1(B(R))$. Besides, $U_n^2 $ satisfies the local energy equality with lower order terms:
\begin{align} \label{leewithlower}
\partial_t|U_n^2|^2 - \Delta |U_n^2|^2 + 2 |\nabla U_n^2|^2 + &
{\rm div}\, (|U_n^2|^2 U_n^2+ |U_n^2|^2 U_n^1) \no \\
& + 2 U_n^2 {\rm div}\,(U_n^2\otimes U_n^1) + 2{\rm div}\,(U_n^2 Q_n^2) = 0,
\end{align}
which can be interpreted in the sense of distributions. Collecting the estimates of $U_n^2$ and $Q_n^2$ and taking the estimates of $U_n^1$ into consideration, we can claim the conclusion below.
\begin{lem} \label{energyerror}
There exist limit functions $u_{\infty}^2$ and $q_{\infty}^2$ defined on $(-2,0) \times \R^d$, satisfying
\begin{itemize}
  \item[{\rm (1)}] $U_n^2 \rightharpoonup u_{\infty}^2 $ \ weakly* in \ $L^{\infty}(-2,0; \dot B^{s_p}_{p,p})$\ and \ $L^{\infty}(-2,0; L^2)$;

  \item[{\rm (2)}] $\nabla U_n^2 \rightharpoonup \nabla u_{\infty}^2 $ \ weakly in \ $L^{2}(-2,0; L^2)$;

  \item[{\rm (3)}] $U_n^2 \rightarrow u_{\infty}^2$\  strongly in \ $L^{\beta}(-2,0;L^{\beta}(B(R)))$ and $C([-2,0];H^{-1}(B(R)))$, \ \ $\forall\, 1\leq \beta <4,\, R>0$;

  \item[{\rm (4)}] $Q_n^2 \rightharpoonup q_{\infty}^2 $\ weakly in \ $L^{2}(-2,0; L^2)$;

  \item[{\rm (5)}] The following local energy inequality
  \begin{align} \label{leiwithlower}
  &\int_{B(R)} |u_{\infty}^2(x,t)|^2 \phi(x,t) dx + 2\int_{-2}^t \int_{B(R)} |\nabla u_{\infty}^2|^2 \phi(x,\tau) dxd\tau  \no \\
  & \leq \int_{-2}^t \int_{B(R)}|u_{\infty}^2|^2(\partial_{\tau} \phi + \Delta \phi) + u_{\infty}^2 \cdot \nabla \phi (|u_{\infty}^2|^2 + 2q_{\infty}^2) + |u_{\infty}^2|^2 u_{\infty}^1 \cdot \nabla \phi  dxd\tau \no \\
   & \qquad \qquad  + 2 \int_{-2}^t \int_{B(R)} (u_{\infty}^1 \cdot u_{\infty}^2) u_{\infty}^2 \cdot \nabla \phi + (u_{\infty}^2 \cdot \nabla u_{\infty}^2) u_{\infty}^1 \phi dxd\tau
  \end{align}
  holds for any $t\in (-2,0)$ and $0\leq \phi \in C_0^{\infty}((-2,0)\times B(R))$, here, $R>0$ is arbitrary.
\end{itemize}
\end{lem}
\begin{proof}
It is easy to see that ${\rm (1),\,(2)}$ and ${\rm (4)}$ follows from~\eqref{globalenergy1}, \eqref{4estimate7}.  The validity of ${\rm (3)}$ can be argued as follows:
\begin{align}
U_n^2 \in L^2(-2,0; H^1(B(R))), \ \ \ \partial_t U_n^2 \in L^2(-2,0; H^{-1}(B(R))),
\end{align}
Appealing to the Aubin-Lions Lemma once again, we obtain
\begin{align}
U_n^2 \rightarrow u_{\infty}^2 \ \ \ \textrm{strongly in }\ \ L^2(-2,0;L^2(B(R))) \ \ \textrm{and} \ \ C([-2,0]; H^{-1}(B(R))).
\end{align}
Then interpolation with~\eqref{4estimate8} leads to the required result. Finally, using the fact that $U_n^2$ satisfies~\eqref{leewithlower} distributionally and the convergence properties of $U_n^1$(see Lemma~\ref{limitoffirst}) and $U_n^2,\,Q_n^2$, one can deduce ${\rm (5)}$.
\end{proof}

Recalling that $v_n = U_n^1 + U_n^2$,  $q_n= Q_n^1 + Q_n^2$. Thanks to  Lemma~\ref{limitoffirst} and Lemma~\ref{energyerror}, one can formulate the limit behavior of
$v_n$ and $q_n$ into the following statement.
\begin{prop} \label{limitpro}
There exist limit functions $v_{\infty},\, q_{\infty}$, defined on domain $(-2,0) \times \R^d$,  with $v_{\infty}= u_{\infty}^1 + u_{\infty}^2$, $q_{\infty}= q_{\infty}^1 + q_{\infty}^{2}$, such that for any $0<\sigma <2,\, R>0$, the following properties
hold.
    \begin{itemize}
      \item[\rm (i)] $v_n \rightharpoonup v_{\infty}$  \textrm{weakly* in }   $L^{\infty}(-2,0; \dot B^{s_p}_{p,p})$ and\
      $L^{\infty}(-2+\sigma,\,0; L^2(B(R)))$;

      \item[\rm (ii)]  $v_n \rightarrow v_{\infty} $ strongly in $L^{\beta}(-2+\sigma,\,0; L^{\beta}(B(R)))$, \ $\forall\,1\leq \beta< 4$;

      \item[\rm (iii)] $\nabla v_n \rightharpoonup \nabla v_{\infty} $ weakly in $L^2(-2+\sigma,\, 0;L^2(B(R)))$;

      \item[\rm (iv)] For every $t \in [-2+\sigma,\,0]$, $
      \psi(x) v_n(x,t) \rightharpoonup \psi(x) v_{\infty}(x,t)$ weakly in $L^2(\R^d)$, and the function $t\mapsto \int_{\R^d} v_{\infty}(x,t)\psi(x) dx\  \in C([-2+\sigma,\,0])$, here $\psi \in C_0^{\infty}(\R^d)$;

      \item[\rm (v)] $q_n \rightharpoonup q_{\infty}$ \ weakly in \ $L^{3/2}(-2+\sigma,\,0; L^{3/2}(B(R)))$;

      \item[\rm (vi)] $v_{\infty},\, q_{\infty} $ forms a pair of suitable weak solution on any bounded domain of $(-2+\sigma,\,0)\times \R^d$.
     \end{itemize}

\end{prop}
\begin{proof}
The verification of ${\rm (i)}-{\rm (v)}$ is straightforward, provided one notice relevant properties of $U_n^1$ and $U_n^2$. Now that $v_n$ and $q_n$ is a pair of smooth solution to NS, so it fulfills the local energy equality, taking $n \to \infty$ and using ${\rm (i)}-{\rm (v)}$, one can find ${\rm (vi)}$ follows.
\end{proof}


\subsection{Proof of Theorem~\ref{mildblowupbesov}}

In this section, we shall prove the blowup criterion for NS in critical Besov space. Let $v_n, q_n, v_{\infty}$, $q_{\infty}$ be the functions constructed in Section~\ref{prioriestimate}. We draw on ideas from~\cite{DoDu09,EsSeSv03},  showing first the limit function $v_{\infty}$ vanishes for some time, then
using the strong convergence property of $v_n$ and an interior estimate of $q_n$ to yield that for some small  $\gamma>0$ and large $n_0$, the pair $v_{n_0},\,q_{n_0}$ verifies the condition of $\epsilon$ regularity criterion on $Q(\gamma)$, thus producing  the boundedness of $u$ at the singular point, which is obviously absurd.  Now we start to  implement this argument.
\begin{prop} \label{vanishlimit}
Let $v_{\infty}, q_{\infty}$ be the limit solution obtained in Proposition~\ref{limitpro}, then
\begin{align}
v_{\infty}(x,t) =0,  \ \ \textrm{for } \ \ t\in (-5/4,0].
\end{align}
\end{prop}
\begin{proof}
Observing that $v_{\infty} = u_{\infty}^1 + u_{\infty}^2$, $q_{\infty} = q_{\infty}^1 + q_{\infty}^{2}$, with
\begin{align}
&  u_{\infty}^1  \in L^{\infty}(-5/3,0; L^m), \ \ u_{\infty}^2 \in L^3(-5/3,0; L^3), \no  \\
& q_{\infty}^1 \in L^{\infty}(-5/3,0; L^{m/2}), \ \   q_{\infty}^{2} \in L^{2}(-5/3,0; L^{2}). \no
\end{align}
So for any  $\epsilon_0, \rho >0$,  there exists some $R_0>0$ large, such that
\begin{align} \label{reguconstant}
\sup_{z\in \mathbb{R}^d}\int_{-5/3}^0 \int_{B(R_0)^c \cap B(z,\rho)} |v_{\infty}|^3 + | q_{\infty}|^{\frac{3}{2}} dxdt < \epsilon_0.
\end{align}
Let $z_0 = (x_0, t_0) \in \R^d\times (-3/2,0)$, $\rho = 1/4$,  then $Q(z_0,\rho) \subset \R^d \times (-5/3,0)$. The value
\begin{align}
\frac{1}{\rho^{d-1}} \int_{Q(z_0,\rho)} |v_{\infty}|^3 + |q_{\infty}|^{\frac{3}{2}} dxdt  \leq C \epsilon_0.
\end{align}
provided $|x_0|>\rho + R_0:=R_1$. Now one can specify $\epsilon_0$ so that $C\epsilon_0 < \tilde{\epsilon}_1$, in view of Corollary~\ref{modifiedversionregu}, we know $v_{\infty}$ is bounded in some neighborhood of $z_0$, and hence
\begin{align}
\sup_{-3/2<t<0}\sup_{B(R_1)^c} |v_{\infty}(x,t)| <\infty.
\end{align}
Upon using the regularity results for linear Stokes systems, one can acquire higher order derivatives estimates
\begin{align}
|\nabla^j v_{\infty}(x,t)| \leq N(j),
\end{align}
with $j \geq 1$ and $(x,t) \in B(2 R_1)^c \times (-5/4,0)$.

Next we show $v_{\infty}(x,0)$ vanishes, one can also refer to the same argument in~\cite{albritton16}.  Due to {\rm (iv)} in Proposition~\ref{limitpro}, we know $v_n(0)\rightharpoonup  v_{\infty}(0)$ in the sense of tempered distribution.  In addition,
 $ u(\cdot, T_*) \in \dot B^{s_p}_{p,p}$, so for any $ \epsilon > 0$, there exists $u_{\epsilon} \in C_{0}^{\infty}(\R^d)$, such that
 \begin{align}
    \|u_{\epsilon}(\cdot) - u(\cdot, T_*)\|_{\dot B^{s_p}_{p,p}} < \epsilon.
 \end{align}
 Let $\varphi $ be a Schwartz function, then
 \begin{align}
  \langle v_n(0), \varphi \rangle &= \lambda_n^{-(d-1)} \langle u(X_0+ x, T_*), \varphi(\lambda_n^{-1}x) \rangle  \no \\
               & = \lambda_n^{-(d-1)} \langle u(X_0+ x, T_*)-u_{\epsilon}(X_0+x), \varphi(\lambda_n^{-1}x)\rangle +
                \lambda_n^{-(d-1)} \langle u_{\epsilon}(X_0+x), \varphi(\lambda_n^{-1}x)\rangle, \no \\
                & \leq \|u(T_*)-u_{\epsilon}\|_{\dot B^{s_p}_{p,p}} \|\varphi\|_{\dot B^{-s_p}_{p',p'}}
                     + \lambda_n \|u_{\epsilon}\|_{L^{\infty}} \|\varphi\|_{L^1} \leq C \epsilon  \no
 \end{align}
 provided  $n$ is sufficiently large.  Hence,
 \begin{align}
  \langle v_{\infty}(0), \varphi \rangle = \lim_{n \to \infty} \langle v_{n}(0), \varphi \rangle  =0.
 \end{align}
 As $\varphi \in \mathcal{S}$ is arbitrary, so $v_{\infty}(0) = 0$, as desired.  Now we denote $\omega_{\infty} = {\rm curl}\,u_{\infty}$, then $\omega_{\infty}$ meets
 the differential inequality
 \begin{align}
   |\partial_t \omega_{\infty} - \Delta \omega_{\infty} |\leq N (|\omega_{\infty}| + |\nabla \omega_{\infty}|)
 \end{align}
 on $ B(0,2R_1)^c \times(-5/4,0] $ and $\omega_{\infty}(x,0) =0$. Applying the backward uniqueness theorem (\cite{EsSeSv03}), we reach
 \begin{align}
   \omega_{\infty}(z) =0 \ \ \textrm{on} \ \ B(0,2R_1)^c \times (-5/4,0].
 \end{align}
We continue to establish the regularity of $v_{\infty}$ on $B(0,2R_1) \times(-5/4,0]$. Note that $v_{\infty} = u_{\infty}^1 + u_{\infty}^2$, and 
\begin{align} \label{bound1}
u_{\infty}^1 \in L^{\infty}(-5/4,0; L^{\infty}),
\end{align}
  thus it is reduced to bound $u_{\infty}^2$.  First, the local energy inequality~\eqref{leiwithlower}  implies the following the global energy inequality:
\begin{align} \label{energyglobal}
& \int_{\R^d} |u_{\infty}^2(t_2)|^2 dx + 2 \int_{t_1}^{t_2}\int_{\R^d} |\nabla u_{\infty}^2|^2 dxdt  \no \\
& \leq \int_{\R^d} |u_{\infty}^2(t_1)|^2 dx +  2\int_{t_1}^{t_2} \int_{\R^d} u_{\infty}^1 \otimes u_{\infty}^2:\nabla u_{\infty}^2 dxdt
\end{align}
with almost every $t_1>-2$ and all $t_1\leq t_2<0$, see~\cite{albritton16} for the proof. On the other hand,
\begin{align}
u_{\infty}^2 \in L^{\infty}(-2,0;\dot B^{s_p}_{p,p}) \cap L^{\infty}(-2,0; L^2).
\end{align}
we can choose $t_1\in (-5/4,0]$ so that that~\eqref{energyglobal} holds and $u^2_{\infty}(t_1)\in L^2\cap \dot B^{s_p}_{p,p}$. Considering the equation below
\begin{equation*}
\begin{cases}
\partial_t \tilde{v} - \Delta  \tilde{v} +  \tilde{v} \cdot\nabla  \tilde{v} + u_{\infty}^1 \cdot \nabla \tilde{v} + \tilde{v} \cdot \nabla u_{\infty}^1 + \nabla \tilde{q} =0,  & \\
{\rm div}\, \tilde{v} =0, & \\
\tilde{v}(x,t_1) = u_{\infty}^2(t_1). &
\end{cases}
\end{equation*}
By a standard Picard iteration procedure, one can construct a mild solution $\tilde{v}$ to the above equation on some interval $(t_1,t_1+ \kappa)$, and
\begin{align} \label{4estimate9}
\sup_{t_1+\sigma<t<t_1+ \kappa}\sup_{x\in \R^d} |\tilde{v}|  \leq C \big(\sigma, \|u_{\infty}^2(t_1)\|_{\dot B^{s_p}_{p,p}}, \|u_{\infty}^1\|_{L^{\infty}(-5/4,0;L^{\infty})} \big), \ \ \ \forall\, 0<\sigma<\kappa/2.
\end{align}
Moreover, the global energy equality
\begin{align}
 \int_{\R^d} |\tilde{v}(t)|^2 dx + 2 \int_{t_1}^{t}\int_{\R^d} |\nabla \tilde{v}|^2 dxd\tau
= \int_{\R^d} |u_{\infty}^2(t_1)|^2 dx +  2\int_{t_1}^{t} \int_{\R^d} u_{\infty}^1 \otimes \tilde{v}:\nabla \tilde{v} dxd\tau
\end{align}
fulfills for  $t_1 \leq t<0$. Then weak-strong uniqueness{\footnote{ Barker~\cite{barker16} showed weak-strong uniqueness of 3D Navier-Stokes equation with initial data in $L^2 \cap \dot B^{s_p}_{p,p}$, see Theorem~\ref{wsunique} for details, whereas his method can still be applied to prove a similar result for the perturbed Navier-Stokes equation $\tilde{v}_{\infty}$ satisfies, where the terms  $u_{\infty}^1\cdot \nabla \tilde{v}$ and $\tilde{v} \cdot \nabla u_{\infty}^1$ don't pose  new difficulties because of  the subcriticality of $u_{\infty}^1$. }} for the equation $\tilde{v}$ solves can infer
\begin{align}
u_{\infty}^2 = \tilde{v}, \ \ \ \textrm{on} \ \  (t_1,t_1+\kappa) \times \R^d.
\end{align}
Recalling~\eqref{bound1}, ~\eqref{4estimate9} and the parabolic regularity result, we can see
 \begin{align}
 \sup_{t_1 + 2\sigma < t< t_1 + \kappa} \sup_{x \in \R^d } |\nabla^k v_{\infty}(x,t)| \leq c(\sigma, k), \ \forall \ k \in \N,
 \end{align}
 Meanwhile, on account of the fact that $\omega_{\infty}(z) =0 $ if $z \in B(0,2R_1)^c \times (t_1 + 2\sigma,t_1+ \kappa)$, one can conclude from the   unique continuation theorem (cf. \cite{EsSeSv03})
 \begin{align}
    \omega =0 \ \ \textrm{on} \ \ \R^d \times (t_1 + 2\sigma, t_1 + \kappa).
 \end{align}
 To summary,
 \begin{align}
  {\rm div}\, v_{\infty} =  {\rm curl}\, v_{\infty} = 0 \ \ \textrm{on} \ \ \R^d \times (t_1 + 2\sigma, t_1 + \kappa), \ \ \forall\, 0<\sigma<\kappa/2.
 \end{align}
Accordingly, $\Delta v_{\infty} =0 $ on the same domain. It follows from Liouville's theorem that $v_{\infty}$ equals to some constant. Owing to~\eqref{reguconstant}, we can assert
\begin{align}
v_{\infty} = 0 \ \ \textrm{on} \ \ \R^d \times (t_1 + 2\sigma, t_1 + \kappa).
\end{align}
Since $\sigma$ can be choosen to be arbitrarily small, then $v_{\infty}(t_1) = 0$, because of  the weak continuity property. However, such $t_1$ exists almost everywhere in $(-5/4,0)$,   upon using weak continuity once again, we finally obtain
 \begin{align}
 v_{\infty} (t) = 0  \ \ \textrm{for} \ \ t \in (-5/4,0].
 \end{align}
 This completes the proof.
  \end{proof}

  Before pushing forward, we give an estimate of the pressure term inside a fixed domain,  which will be used later. Estimate of this type can also be found in~\cite{DoDu09,DoDu07}.
\begin{lem} \label{quacd}
Let $0 < \gamma \leq  1/4,\, \rho>0$. $u, p$  forms a pair of  weak solution to {\rm NS} on $Q(\rho)$.  Set $r = \gamma \rho $,
then there exists a constant $C$ independent of $\gamma$, such that
   \begin{align} \label{interpressure}
   \frac{1}{r^{d-1}} \int_{Q(r)} |p|^{\frac{3}{2}} dxdt \leq C \gamma \bigg [\frac{1}{\rho^{d-1}} \int_{Q(\rho)} |p|^{\frac{3}{2}} dxdt \bigg] +
    C \gamma^{-(d-1)} \bigg [\frac{1}{\rho^{d-1}} \int_{Q(\rho)} |u|^{3} dxdt \bigg].
   \end{align}

\end{lem}
\begin{proof}
Let $\phi \in C_0^{\infty}(\R^d)$ be such that ${\rm supp}\,\phi \subset B(1)$  and $\phi =1 $ on $B(1/2)$. Define $\phi_{\rho}(x) = \phi(x/\rho)$, we decompose the pressure as $p = p_{\rho} + h_{\rho} $ on $Q(\rho)$, where
\begin{align}
- \Delta p_{\rho} = \partial_i \partial_j (u_i u_j \phi_{\rho}) \ \ \textrm{on} \ \ \R^d
\end{align}
for a.e. $t\in (-\rho^2,0)$. The other part $h_{\rho}$ is harmonic on $B(\rho/2)$, so we have
\begin{align} \label{estimateofhar}
\sup_{x\in B(r)} |h_{\rho}(x)| &\leq \sup_{x\in B(r)} \frac{1}{|B(x,\rho/4)|} \int_{B(x,\rho/4)}| h_{\rho}(y)| dy  \no \\
& \leq C \frac{1}{\rho^d} \int_{B(\rho/2)} | h_{\rho}(y)| dy.
\end{align}
It follows
\begin{align}
\int_{B(r)} |p|^{\frac{3}{2}} dx \leq C \int_{B(r)} |p_{\rho}|^{\frac{3}{2}} + |h_{\rho}|^{\frac{3}{2}} dx.
\end{align}
By Calderon-Zygmund's estimate
\begin{align}
\int_{B(r)} |p_{\rho}|^{\frac{3}{2}}  dx \leq \int_{\R^d} |p_{\rho}|^{\frac{3}{2}}  dx  \leq  C \int_{B(\rho)} |u|^3 dx.
\end{align}
While according~\eqref{estimateofhar},
\begin{align}
\int_{B(r)}|h_{\rho}|^{\frac{3}{2}} dx \leq C r^{d} \sup_{x\in B(r)} |h_{\rho}(x)|^{\frac{3}{2}} & \leq C \frac{r^d}{\rho^d} \int_{B(\rho/2)} | h_{\rho}(y)|^{\frac{3}{2}} dy  \no \\
& \leq C \frac{r^d}{\rho^d} \bigg[\int_{B(\rho)} |p(y)|^{\frac{3}{2}} dy +  \int_{B(\rho)} |p_{\rho}(y)|^{\frac{3}{2}} dy \bigg].
\end{align}
Hence
\begin{align}
\int_{B(r)} |p|^{\frac{3}{2}} dx \leq C \int_{B(\rho)} |u|^3 dx + C \frac{r^d}{\rho^d} \int_{B(\rho)} |p(y)|^{\frac{3}{2}} dy.
\end{align}
Integrating in time on interval $(-r^2, 0)$ and multiplying each side by $1/r^{d-1}$,  one can obtain~\eqref{interpressure}. The proof is done.
\end{proof}

  We are ready to prove Theorem~\ref{mildblowupbesov}.
  \begin{proof}[Proof of Theorem~\ref{mildblowupbesov}]
    We claim that there exist some $0<\gamma<1$ sufficiently small and an index $n_0$ sufficiently large, so that
   \begin{align} \label{verifysmall}
   \frac{1}{\gamma^{d-1}}\int_{Q(\gamma)} |v_{n_0}|^3 + |q_{n_0}|^{\frac{3}{2}} dxdt < \tilde{\epsilon}_1.
   \end{align}
where $\tilde{\epsilon}_1$ is determined by~\eqref{scalesmaconst} and depends only on $d,\,\widetilde{M}$.
Indeed, it follows from Lemma~\ref{quacd} with $\rho =1$ that
   \begin{align}
   \frac{1}{\gamma^{d-1}}\int_{Q(\gamma)}  |q_{n}|^{\frac{3}{2}} dxdt \leq C \gamma \int_{Q(1)} |q_n|^{\frac{3}{2}} dxdt  +
    C \gamma^{-(d-1)} \int_{Q(1)} |v_n|^{3} dxdt.
   \end{align}
Due to Proposition~\ref{limitpro} and Proposition~\ref{vanishlimit}, we have for some constant $C$, it holds
   \begin{align} \label{vanish}
   \int_{Q(1)} |q_n|^{\frac{3}{2}} dxdt \leq C, \ \ \ \lim_{n\to \infty} \int_{Q(1)} |v_n|^3 dxdt = 0.
   \end{align}
Thus one can choose $\gamma$ small enough so that
   \begin{align}
   C \gamma \int_{Q(1)} |q_n|^{\frac{3}{2}} dxdt < \frac{\tilde{\epsilon}_1}{4}.
   \end{align}
Fix such $\gamma$, by~\eqref{vanish}, there exists some $n_0$ large, satisfying
\begin{align}
C \frac{1}{\gamma^{d-1}} \int_{Q(1)} |v_{n_0}|^{3} dxdt < \frac{\tilde{\epsilon}_1}{4}.
\end{align}
In this way, we find~\eqref{verifysmall} follows. Meanwhile, $v_{n_0},\,q_{n_0}$ is a pair of suitable weak solution on $Q(\gamma)$, now applying Corollary~\ref{modifiedversionregu}, one readily obtains
\begin{align}
\sup_{(x,t)\in Q(\gamma/4)} |v_{n_{0}}(x,t)| \leq C,
\end{align}
or in terms of $u$, we have
\begin{align}
\sup_{(x,t)\in Q(Z_0,\, \gamma \lambda_{n_0}/ {4}  )} |u(x,t)| \leq C \lambda_{n_0}^{-1},
\end{align}
which obviously contradicts to our hypothesis that $Z_0$ is a singular point, hence, the conclusion of Theorem~\ref{mildblowupbesov} is true. This completes the  proof.
  \end{proof}


\section{Leray-Hopf solution in critical Besov space} \label{lerayhopfregu}

The objective of this section is to  show an endpoint Serrin type regularity criterion for Leray-Hopf solution, i.e. Corollary~\ref{leraybesov}. Before stating the precise notion of Leray-Hopf solution, we first clarify some necessary notations being used, let $\dot C_0^{\infty}(\R^d):= \{u \in C_0^{\infty}(\R^d):\ {\rm div}\,u = 0\}$, $\dot J$ and $\dot J^1_2$ represent the closure of $\dot C_0^{\infty}$ in the norm of $L^2$ and Dirichlet integral respectively.
\begin{defn} [Leray-Hopf solution] \label{weakleray}
Let $T \in (0,\infty]$, a vector field $v$ is said to be a Leray-Hopf solution to~\eqref{ns} on $Q_T: = (0,T) \times \R^d$ if
\begin{itemize}
  \item[\rm (i)] $v \in L^{\infty}(0,T;\dot J) \cap L^2(0,T; \dot J_2^1)$  satisfies {\rm NS } distributionally on $Q_T$;
  \item[\rm (ii)] For each $g \in L^2$, the function $t \mapsto \int_{\R^d} v(x,t) g(x) dx $ is continuous on $[0,T]$;
  \item[\rm (iii)] The global energy inequality
    \begin{align}
      \|v(t)\|^2_{L^2} + 2 \int_0^t \|\nabla v(s)\|_{L^2}^2 ds  \leq \|v_0\|_{L^2}^2
    \end{align}
    holds for each $t \in  [0,T]$.
\end{itemize}
\end{defn}
Notably, in the above definition, we call $v$ a global Leray-Hopf solution if $T =\infty$ and the interval $[0,T]$ is replaced by $[0,\infty)$ in {\rm (ii)} and {\rm (iii)}.

The following weak-strong uniqueness result shows the connection between the Leray-Hopf solution and mild solution in critical Besov space, and plays an important role in the upcoming proof. We point out that its three dimensional counterpart  is  contained in~\cite{barker16}, where the proof  can be  adapted to higher dimension without too many difficulties.
\begin{thm} [Weak-Strong uniqueness] \label{wsunique}
Let $3\leq d<p,\,q < \infty$,  $u$ be a Leray-Hopf solution to~\eqref{ns} associated with initial data $u_0 \in L^2(\R^d) \cap \dot B^{s_p}_{p,q}(\R^d)$, then $u$ coincides with the mild solution $NS(u_0)$ until $T(u_0)$, in particular, $u$ is regular for the same time interval.
\end{thm}

\begin{proof} [Proof of Corollary~\ref{leraybesov}]

Let $d<p,\,q<\infty$, $u \in L^{\infty}(0,T;\dot B^{s_p}_{p,q})$ be a Leray-Hopf solution. Due to the weak continuity, one can deduce that the initial data $u_0$ satisfies
\begin{align}
u_0 \in L^2 \cap \dot B^{s_p}_{p,q}.
\end{align}
Employing Theorem~\ref{wsunique}, one can see
\begin{align}
u = NS(u_0) \ \ \ \textrm{on} \ \ \R^d \times [0,T(u_0)).
\end{align}
We claim  that
\begin{align} \label{exitencetime}
T(u_0)> T.
\end{align}
Otherwise, if $T(u_0) \leq T$, then  Theorem~\ref{mildblowupbesov} implies
\begin{align}
\limsup_{t\to T(u_0)} \|u(t)\|_{\dot B^{s_p}_{p,q}} = \infty.
\end{align}
This is contrary to our hypothesis, so~\eqref{exitencetime} holds. Since the mild solution $NS(u_0)$ is smooth on $\R^d \times (0,T(u_0))$, so does $u$. The uniqueness follows immediately from Theorem~\ref{wsunique}.  We complete the  proof.
\end{proof}

\noindent{\bf Acknowledgment.} Both of the authors were
supported in part by the National Science Foundation of China, grants 11271023 and 11771024.  The first named author is grateful to Professor F. Planchon for his valuable discussions when he was visiting Laboratoire J.~A.~Dieudonn\'e.

\end{document}